\documentclass[leqno,11pt,twoside]{amsart}
\usepackage[body={6.5in, 8.5in},left=1.3in,right=1.3in]{geometry}

\usepackage{times}

\usepackage[all]{xy} 
\usepackage{amsmath, amssymb, amsfonts, latexsym, mdwlist, amsthm}
\usepackage{subfig}
\usepackage{graphicx}
\usepackage{wrapfig}

\usepackage{url}

\usepackage[bookmarks, colorlinks, breaklinks, pdftitle={Stability conditions on abelian threefolds},
pdfauthor={Arend Bayer, Emanuele Macri, Paolo Stellari}]{hyperref}
\hypersetup{linkcolor=blue,citecolor=blue,filecolor=black,urlcolor=blue}

\usepackage{pgf, tikz}
\usetikzlibrary{arrows,chains,shapes.geometric,%
    decorations.pathreplacing,decorations.pathmorphing,shapes,%
    matrix,shapes.symbols}

\tikzset{
>=stealth',
  punktchain/.style={
    rectangle,
    rounded corners,
    draw=black, thick,
    minimum height=3em,
    text centered,
    on chain},
  line/.style={draw, thick, <-},
  element/.style={
    tape,
    top color=white,
    bottom color=blue!50!black!60!,
    minimum width=8em,
    draw=blue!40!black!90, very thick,
    text width=10em,
    minimum height=3.5em,
    text centered,
    on chain},
  every join/.style={->, thick,shorten >=1pt},
  decoration={brace},
  tuborg/.style={decorate},
  tubnode/.style={midway, right=2pt},
}

\usepackage{paralist}
\setdefaultenum{(a)}{(i)}{}{}
\usepackage{enumitem} 

\def\C{\ensuremath{\mathbb{C}}}

\def\H{\ensuremath{\mathbb{H}}}
\def\N{\ensuremath{\mathbb{N}}}
\def\P{\ensuremath{\mathbb{P}}}
\def\Q{\ensuremath{\mathbb{Q}}}
\def\R{\ensuremath{\mathbb{R}}}
\def\Z{\ensuremath{\mathbb{Z}}}

\def\ch{\mathop{\mathrm{ch}}\nolimits}

\def\Coh{\mathop{\mathrm{Coh}}\nolimits}

\def\dim{\mathop{\mathrm{dim}}\nolimits}

\def\ext{\mathop{\mathrm{ext}}\nolimits}
\def\Ext{\mathop{\mathrm{Ext}}\nolimits}
\def\GL{\mathop{\mathrm{GL}}\nolimits}

\def\Hom{\mathop{\mathrm{Hom}}\nolimits}
\def\id{\mathop{\mathrm{id}}\nolimits}

\def\Ker{\mathop{\mathrm{Ker}}\nolimits}

\def\NS{\mathop{\mathrm{NS}}\nolimits}

\def\rk{\mathop{\mathrm{rk}}}

\def\Tor{\mathop{\mathrm{Tor}}\nolimits}

\def\MG13{\ensuremath{{\mathcal M}_{\Gamma_1(3)}}}
\def\tildeMG13{\ensuremath{\widetilde{\mathcal M}_{\Gamma_1(3)}}}
\def\Stab{\mathop{\mathrm{Stab}}\nolimits}
\def\into{\ensuremath{\hookrightarrow}}
\def\onto{\ensuremath{\twoheadrightarrow}}

\def\blank{\underline{\hphantom{A}}}

\def\Db{\mathrm{D}^{b}}

\newcommand\TFILTB[3]{%
\xymatrix@=1pc{
{0 = {#1}_0} \ar[rr]&&
{{#1}_1} \ar[rr]\ar[ld] &&
{{#1}_2} \ar[r]\ar[ld] &
{\cdots} \ar[r] & { {#1}_{#3-1}} \ar[rr] &&
{{#1}_{#3} = {#1}} \ar[ld]
\\
& *{{#2}_1} \ar@{.>}[ul] &&
{{#2}_2} \ar@{.>}[ul] & &&&
{{#2}_{{#3}}} \ar@{.>}[ul]
}}

\def\abs#1{\left\lvert#1\right\rvert}

\newcommand\stv[2]{\left\{#1\,\colon\,#2\right\}}

\makeatletter
\newtheorem*{rep@theorem}{\rep@title}
\newcommand{\newreptheorem}[2]{%
\newenvironment{rep#1}[1]{%
 \def\rep@title{#2 \ref{##1}}%
 \begin{rep@theorem}}%
 {\end{rep@theorem}}}
\makeatother

\newtheorem{Thm}{Theorem}[section]
\newreptheorem{Thm}{Theorem}
\newtheorem{Prop}[Thm]{Proposition}

\newtheorem{Lem}[Thm]{Lemma}
\newtheorem{Cor}[Thm]{Corollary}
\newreptheorem{Cor}{Corollary}
\newtheorem{Con}[Thm]{Conjecture}
\newreptheorem{Con}{Conjecture}

\newtheorem{thm-int}{Theorem}

\theoremstyle{definition}
\newtheorem{Def-s}[Thm]{Definition}
\newtheorem{Def}[Thm]{Definition}
\newtheorem{Rem}[Thm]{Remark}

\newtheorem{Ex}[Thm]{Example}

\def\C{\ensuremath{\mathbb{C}}}

\def\H{\ensuremath{\mathbb{H}}}
\def\N{\ensuremath{\mathbb{N}}}
\def\P{\ensuremath{\mathbb{P}}}
\def\Q{\ensuremath{\mathbb{Q}}}
\def\R{\ensuremath{\mathbb{R}}}
\def\Z{\ensuremath{\mathbb{Z}}}

\def\AA{\ensuremath{\mathcal A}}
\def\BB{\ensuremath{\mathcal B}}
\def\CC{\ensuremath{\mathcal C}}
\def\DD{\ensuremath{\mathcal D}}

\def\FF{\ensuremath{\mathcal F}}

\def\OO{\ensuremath{\mathcal O}}
\def\PP{\ensuremath{\mathcal P}}

\def\QQ{\ensuremath{\mathcal Q}}
\def\TT{\ensuremath{\mathcal T}}
\def\UU{\ensuremath{\mathcal U}}
\def\VV{\ensuremath{\mathcal V}}

\def\ZZ{\ensuremath{\mathcal Z}}

\def\CCC{\mathfrak C}

\def\PPP{\mathfrak P}

\def\VVV{\mathfrak V}

\newcommand{\ignore}[1]{}

\def\um{\underline{m}}
\def\un{\underline{n}}
\def\uq{\underline{q}}
\def\obeta{\overline{\beta}}

\def\HDelta{\overline{\Delta}_H}
\def\HDeltaB{\overline{\Delta}_H^B}
\def\HDeltaBC{\Delta_{H, B}^C }
\def\HNablab{\overline{\nabla}_H^\beta}

\begin{document}

\title[Stability Conditions on Abelian Threefolds and some Calabi-Yau threefolds]{The Space of Stability Conditions on Abelian Threefolds, and on some Calabi-Yau Threefolds}

\author{Arend Bayer}
\address{School of Mathematics and Maxwell Institute,
University of Edinburgh,
James Clerk Maxwell Building,
Peter Guthrie Tait Road, Edinburgh, EH9 3FD,
United Kingdom}
\email{arend.bayer@ed.ac.uk}
\urladdr{\url{http://www.maths.ed.ac.uk/~abayer/}}

\author{Emanuele Macr\`i}
\address{Department of Mathematics, The Ohio State University, 231 W 18th Avenue, Columbus, OH
43210, USA}
\curraddr{Department of Mathematics, Northeastern University, 360 Huntington Avenue, Boston, MA
02115, USA}
\email{e.macri@neu.edu}
\urladdr{\url{http://nuweb15.neu.edu/emacri/}}

\author{Paolo Stellari}
\address{Dipartimento di Matematica ``F. Enriques'', Universit\`a degli Studi di Milano, Via Cesare Saldini 50, 20133 Milano, Italy}
\email{paolo.stellari@unimi.it}
\urladdr{\url{http://users.unimi.it/stellari}}



\begin{abstract}
We describe a connected component of the space of stability conditions on abelian threefolds, and on
Calabi-Yau threefolds obtained as (the crepant resolution of) a finite quotient of an abelian
threefold. Our proof includes the following essential steps:

1. We simultaneously strengthen a conjecture by the first two authors and Toda, and prove that it
follows from a more natural and seemingly weaker statement. This conjecture is a Bogomolov-Gieseker
type inequality involving the third Chern character of ``tilt-stable'' two-term complexes on
smooth projective threefolds; we extend it from complexes of tilt-slope zero to arbitrary tilt-slope.

2. We show that this stronger conjecture implies the so-called support property of Bridgeland
stability conditions, and the existence of an explicit open subset of the space of stability
conditions.

3. We prove our conjecture for abelian threefolds, thereby reproving and generalizing
a result by Maciocia and Piyaratne.

Important in our approach is a more systematic understanding on the behaviour of quadratic
inequalities for semistable objects under wall-crossing, closely related to the support
property.
\end{abstract}


\maketitle
\setcounter{tocdepth}{1}
\tableofcontents

\section{Introduction}\label{sec:intro}
In this paper, we determine the space of Bridgeland stability conditions on abelian
threefolds and on Calabi-Yau threefolds obtained either as a finite quotient of an abelian threefold, or as the crepant resolution of such a quotient. More precisely, we describe a connected
component of the space of stability conditions for which the central charge
only depends on the degrees $H^{3-i} \ch_i(\blank)$, $i = 0, 1, 2, 3$, of the Chern
character\footnote{In the case of crepant resolutions, we take the Chern character after applying BKR-equivalence \cite{Mukai-McKay} between the crepant resolution and the orbifold quotient.}
with respect to a given polarization $H$, and that satisfy the support property.

\subsection*{Stability conditions on threefolds via a conjectural Bogomolov-Gieseker type inequality}
The existence of stability conditions on three-dimensional
varieties in general, and more specifically on Calabi-Yau threefolds, is often considered the
biggest open problem in the theory of Bridgeland stability conditions. Until recent work by
Maciocia and Piyaratne \cite{Dulip-Antony:I, Dulip-Antony:II}, they were
only known to exist on threefolds whose derived category admits a full exceptional collection.
Possible applications of stability conditions range from modularity properties of generating
functions of Donaldson-Thomas invariants \cite{Toda:DenefMoore, Yukinobu:ICM} to Reider-type theorems for
adjoint linear series \cite{BBMT:Fujita}.

In \cite{BMT:3folds-BG}, the first two authors and Yukinobu Toda, also based on discussions with
Aaron Bertram, proposed a general approach towards the
construction of stability conditions on a smooth projective threefold $X$. The construction is based
on the auxiliary notion of \emph{tilt-stability} for two-term complexes, and a conjectural
Bogomolov-Gieseker type inequality for the third Chern character of tilt-stable objects; we review
these notions  in Section \ref{sec:TiltStability} and the precise inequality in Conjecture
\ref{con:strongBG}. It depends on the choice of two divisor classes $\omega, B \in \NS(X)_\R$
with $\omega$ ample.
It was shown that this conjecture would imply the existence of Bridgeland
stability conditions\footnote{Not including the so-called ``support property'' reviewed further
below.}, and, in the companion paper \cite{BBMT:Fujita}, a version of an open case of Fujita's
conjecture, on the very ampleness of adjoint line bundles on threefolds.

Our first main result is the following, generalizing the result of
\cite{Dulip-Antony:I, Dulip-Antony:II} for the case when $X$ has Picard rank one:
\begin{Thm}\label{thm:mainabelian}
The Bogomolov-Gieseker type inequality for tilt-stable objects,
Conjecture \ref{con:strongBG}, holds when $X$ is an abelian threefold, and $\omega$ is a real
multiple of an integral ample divisor class.
\end{Thm}

There are Calabi-Yau threefolds that admit an abelian variety as a finite \'etale cover; we call
them \emph{Calabi-Yau threefolds of abelian type}. Our result applies similarly in these
cases:

\begin{Thm}\label{thm:mainCYabelian}
Conjecture \ref{con:strongBG} holds when $X$ is a Calabi-Yau threefold of abelian type, and $\omega$ is a real
multiple of an integral ample divisor class.
\end{Thm}

Combined with the results of \cite{BMT:3folds-BG}, these theorems imply the existence of Bridgeland
stability conditions in either case. There is one more type of Calabi-Yau threefolds whose
derived category is closely related to those of abelian threefolds: namely \emph{Kummer threefolds},
that are obtained as the crepant resolution of the quotient of an abelian threefold $X$ by the action
of a finite group $G$. Using the method of ``inducing'' stability conditions on the $G$-equivariant
derived category of $X$ and the BKR-equivalence \cite{Mukai-McKay}, we can also treat this case.  Overall this leads to
the following result (which we will make more precise in Theorem \ref{MAINTHM}).

\begin{Thm} \label{maintheorem}
Bridgeland stability conditions on $X$ exist when $X$ is an abelian threefold, or a Calabi-Yau
threefold of abelian type, or a Kummer threefold.
\end{Thm}

\subsection*{Support property}
The notion of support property of a Bridgeland stability condition is crucial in order to apply the
main result of \cite{Bridgeland:Stab}, namely that the stability condition can be deformed;
moreover, it ensures that the space of such stability conditions satisfies well-behaved wall-crossing.

In order to prove the support property, we first need a quadratic inequality for all
tilt-stable complexes, whereas
Conjecture \ref{con:strongBG} only treats complexes $E$ with tilt-slope zero.
We state such an inequality in Conjecture \ref{con:BGgeneral} for the case where $\omega, B$ are
proportional to a given ample class $H$:

\begin{repCon}{con:BGgeneral}
Let $(X, H)$ be a smooth polarized threefold, and
$\omega = \sqrt{3} \alpha H$, $B = \beta H$, for $\alpha>0$, $\beta\in\R$.
If $E\in\Db(X)$ is tilt-semistable with respect to $\omega, B$, then
\[
\alpha^2 \Bigl( \left(H^2\ch_1^B(E)\right)^2 - 2 H^3 \ch_0^B(E) H\ch_2^B(E)\Bigr) + 4 \left( H \ch_2^B(E)\right)^2 - 6 H^2 \ch_1^B(E) \ch_3^B(E) \ge 0,
\]
where $\ch^B := e^{-B}\ch$.
\end{repCon}

In Theorem \ref{thm:BGgeneral}, we prove that this generalized conjecture
is in fact equivalent to the original Conjecture \ref{con:strongBG}. Moreover,
in Theorem \ref{thm:BGinDim3} we prove that it implies a similar quadratic inequality for objects that are stable with respect
to the Bridgeland stability conditions constructed in Theorem \ref{maintheorem}, thereby obtaining a
version of the support property.

To be precise, we consider stability conditions whose
\emph{central charge} $Z \colon K(X) \to \C$ factors via
\begin{equation} \label{eq:chvector}
v_H\colon K(X) \to \Q^4, \quad E \mapsto \left(H^3 \ch_0(E), H^2 \ch_1(E), H \ch_2(E),
\ch_3(E)\right).
\end{equation}
(In the case of Kummer threefolds, we apply the BKR-equivalence before
taking the Chern character.) We prove the support property with respect to $v_H$; this shows
that a stability condition deforms along a small deformation of its central charge, if that
deformation still factors via $v_H$.

We discuss the relation between support property, quadratic inequalities for semistable objects and
deformations of stability conditions systematically in Appendix \ref{app:SupportProperty}. In
particular, we obtain an explicit open subset of stability conditions whenever Conjecture
\ref{con:BGgeneral} is satisfied, see Theorem \ref{thm:mainPPP}.

\subsection*{The space of stability conditions}
In each of the cases of Theorem \ref{maintheorem}, we show moreover that this open subset is
a connected component of the space of stability conditions. We now give a description of this
component.

Inside the space $\Hom(\Q^4, \C)$, consider the open subset $\VVV$ of linear maps $Z$ whose kernel
does not intersect the (real) twisted cubic $\CCC\subset \P^3(\R)$ parametrized by $(x^3, x^2y,
\frac 12 xy^2, \frac 16 y^3)$; it is the complement of a real hypersurface.  Such a linear map $Z$
induces a morphism $\P^1(\R) \cong \CCC \to \C^*/\R^* = \P^1(\R)$; we define $\PPP$ be the component
of $\VVV$ for which this map is an unramified cover of topological degree $+3$ with respect to the
natural orientations. Let $\widetilde \PPP$ be its universal cover.

We let $\Stab_H(X)$ be the space of stability conditions for which the central charge factors via the
map $v_H$ as in equation \eqref{eq:chvector} (and satisfying the support property).

\begin{Thm}  \label{MAINTHM}
Let $X$ be an abelian threefold, or a Calabi-Yau threefold of abelian type, or a Kummer threefold. Then $\Stab_H(X)$ has a
connected component isomorphic to $\widetilde \PPP$.
\end{Thm}

\subsection*{Approach} We will now explain some of the key steps of our approach.

\subsubsection*{Reduction to a limit case}
The first step applies to any smooth projective threefold. Assume that $\omega,
B$ are proportional to a given ample polarization $H$ of $X$. We reduce Conjecture
\ref{con:BGgeneral} to a statement for objects $E$ that are stable in the limit as
$\omega(t) \to 0$ and $\nu_{\omega(t), B(t)}(E) \to 0$; if $\overline{B} := \lim B(t)$, the claim is
that
\begin{equation} \label{eq:ch3Ble0}
\int_X e^{-\overline B}\ch(E) \le 0.
\end{equation}
The reduction is based on the methods of \cite{Macri:P3}: as we approach this limit, either $E$
remains stable, in which case the above inequality is enough to ensure that $E$ satisfies our
conjecture everywhere. Otherwise, $E$ will be strictly semistable at some point; we then show that
all its Jordan-H\"older factors have strictly smaller ``\emph{$H$-discriminant}'' (which is a
variant of the discriminant appearing in the classical Bogomolov-Gieseker inequality). This
allows us to proceed by induction.

\subsubsection*{Abelian threefolds}
In the case of an abelian threefold, we make extensive use of the multiplication by $m$ map
$\underline{m}\colon X \to X$ in order to establish inequality \eqref{eq:ch3Ble0}. The key fact is that if $E$ is tilt-stable, then so is
$\underline{m}^*E$.

To illustrate these arguments, assume that $\overline{B}$ is rational. Via pull-back we can then
assume that $\overline{B}$ is integral; by tensoring with $\OO_X(\overline B)$ we reduce to the case of
$\overline{B} = 0$. We then have to prove that $\ch_3(E) \le 0$; in other words, we have to prove
an inequality of the Euler characteristic of $E$. To obtain a contradiction, assume that
$\ch_3(E) > 0$, and consider further pull-backs:
\begin{equation} \label{eq:ch3m6}
\chi(\OO_X, \underline{m}^*E) = \ch_3(\underline{m}^*E) = m^6 \ch_3(E) \ge m^6.
\end{equation}
However, by stability we have $\Hom(\OO_X(H), \underline{m}^*E) = 0$; moreover, if $D \in \abs{H}$ is a
general element of the linear system of $H$, classical arguments, based on the Grauert-M\"ulich
theorem and bounds for global sections of slope-semistable sheaves, give a bound of the form
\[
h^0(\underline{m}^*E) \le h^0((\underline{m}^*E)|_D) = O(m^4)
\]
Similar bounds for $h^2$ lead to a contradiction to \eqref{eq:ch3m6}.

\subsubsection*{Support property}
As pointed out by Kontsevich and Soibelman in \cite[Section 2.1]{Kontsevich-Soibelman:stability},
the support property is equivalent to the existence of a real quadratic form $Q \colon \Q^4 \to
\R$ such that
\begin{enumerate}
\item \label{enum:negativeonkernel} the kernel of the central charge (as a subspace of $\R^4$) is negative
definite with respect to $Q$, and
\item \label{enum:positiveonobjects} every semistable object $E$ satisfies $Q(v_H(E)) \ge 0$.
\end{enumerate}
The inequality in Conjecture \ref{con:BGgeneral} precisely gives such a quadratic form. We therefore
need to show that this inequality is preserved when we move from tilt-stability to actual Bridgeland
stability conditions.

We establish a more basic phenomenon of this principle in Appendix \ref{app:SupportProperty}, which
may be of independent interest: if a stability condition satisfies the support property with
respect to $Q$, and if we deform along a path for which the central charges all satisfy condition
\eqref{enum:negativeonkernel}, then condition \eqref{enum:positiveonobjects} remains preserved under
this deformation, i.e., it is preserved under wall-crossing. The essential arguments involve elementary
linear algebra of quadratic forms.

Tilt-stability can be thought of as a limiting case of a path in the set of stability conditions we
construct. In Section \ref{sec:newstability} we show that the principle described in the previous
paragraph similarly holds in this case: we show that a small perturbation of the quadratic form in Conjecture
\ref{con:BGgeneral} is preserved under the wall-crossings between tilt-stability and any of our
stability conditions, thereby establishing the desired support property.

\subsubsection*{Connected component}
In Appendix \ref{app:SupportProperty}, we also provide a more effective version of Bridgeland's
deformation result. In particular, the proof of the support property yields large open sets of
stability conditions, which combine to cover the manifold $\widetilde \PPP$ described above.

In Section \ref{sec:abelianspace}, we show that this set is in fact an entire component.
The proof is based on the observation that semi-homogeneous vector bundles $E$ with $c_1(E)$
proportional to $H$ are stable everywhere on $\PPP$; their Chern classes (up to rescaling) are dense
in $\CCC$.

This fact is very unique to varieties admitting \'etale covers by abelian threefolds.
In particular, while Conjecture \ref{con:BGgeneral} implies that $\widetilde \PPP$ is a subset of
the space of stability conditions, one should in general expect the space to be much larger than
this open subset.

\subsection*{Applications}
Our work has a few immediate consequences unrelated to derived categories. Although these are fairly
specific, they still serve to illustrate the power of Conjecture \ref{con:BGgeneral}.

\begin{Cor}
Let $X$ be a Calabi-Yau threefold of abelian type.
Given $\alpha \in \Z_{>0}$, let $L$ be an ample line bundle on $X$ satisfying
\begin{itemize}
\item $L^3 > 49\alpha$,
\item $L^2 D \ge 7\alpha$ for every integral divisor class $D$ with $L^2 D > 0$
and $L D^2 < \alpha$, and
\item $L.C \ge 3\alpha$ for every curve $C \subset X$.
\end{itemize}
Then
$ H^1(L \otimes I_Z) = 0 $
for every 0-dimensional subscheme $Z \subset X$ of length $\alpha$.

In addition, if $L = A^{\otimes 5}$ for an ample line bundle $A$, then $L$ is very ample.
\end{Cor}
\begin{proof} Since Conjecture \ref{con:strongBG} holds for $X$ by our Theorem \ref{thm:mainCYabelian}, we
can apply Theorem 4.1 and Remark 4.3 of \cite{BBMT:Fujita}.
\end{proof}
Setting $\alpha = 2$ we obtain a Reider-type criterion for $L$ to be very ample.
The statement for $A^{\otimes 5}$ confirms (the very ampleness case of) Fujita's conjecture for such $X$.
The best known bounds for Calabi-Yau threefolds say that
$A^{\otimes 8}$ is very ample if $L^3 > 1$ \cite[Corollary 1]{Gallego-Purnaprajna:CY3folds},
$A^{\otimes 10}$ is very ample in general, and that $A^{\otimes 5}$ induces a birational map
\cite[Theorem I]{Oguiso-Peternell:CY3folds}.
For abelian varieties, much stronger statements are known, see \cite{PP:regularity-I, PP:regularity-II}.

\begin{Cor} \label{cor:slopestable}
Let $X$ be one of the following threefolds: projective space, the quadric in $\P^4$, an abelian
threefold, or a Calabi-Yau threefold of abelian type. Let $H$ be a polarization, and let
$c \in \Z_{>0}$ be the minimum positive value of $H^2 D$ for integral divisor classes $D$.
If $E$ is a sheaf that is slope-stable with respect to $H$, and with
$H^2 c_1(E) = c$, then
\[
3 c \ch_3(E) \le 2 \left(H \ch_2(E)\right)^2.
\]
\end{Cor}
The assumptions hold when $\NS(X)$ is generated by $H$, and $c_1(E) = H$. We
refer to Example \ref{ex:slopeCorollary} and Remark \ref{rem:slopeCorollary} for a
proof and more discussion.  Even for vector bundles on $\P^3$, this statement was not previously known for rank bigger
than three.

It is a special case of Conjecture \ref{con:BGgeneral}. Even when $X$ is a complete intersection
threefold and $E = I_C \otimes L$  is the twist of an ideal sheaf of a curve $C$, this inequality is
not known, see \cite{Becca:Castelnuovo}.

\subsection*{Open questions}
\subsubsection*{General proof of Conjecture \ref{con:BGgeneral}} While Conjecture
\ref{con:BGgeneral} for arbitrary threefolds remains elusive, our approach
seems to get a bit closer: in our proof of Theorem \ref{thm:mainabelian}
(in Sections \ref{sec:TiltStability}---\ref{sec:abelianBG0}), only Section \ref{sec:abelianBG0} is
specific to abelian threefolds. One could hope to generalize our construction by replacing the
multiplication map $\um$ with ramified coverings.  This would immediately
yield the set $\widetilde \PPP$ as an open subset of the space of stability conditions.

\subsubsection*{Strengthening of Conjecture \ref{con:BGgeneral}} In order to construct a set of
stability conditions of dimension equal to the rank of the algebraic cohomology of $X$, we would
need a stronger Bogomolov-Gieseker type inequality, depending on $\ch_1$ and $\ch_2$
directly, not just on $H^2 \ch_1$ and $H \ch_2$. We point out that the obvious guess,
namely to replace $\left(H^2 \ch_1\right)^2$ by $H \ch_1^2 \cdot H^3$, and $\left(H \ch_2\right)^2$
by an appropriate quadratic form on $H^4(X)$, does not work in general: for $\alpha \to +\infty$, such an
inequality fails for torsion sheaves supported on a divisor $D$ with
$H D^2 < 0$.

\subsubsection*{Higher dimension} Our work also clarifies the expectations for higher dimensions.
The definition of $\PPP$ directly generalizes to
dimension $n$ in an obvious way, by replacing the twisted cubic with the rational normal curve
$\left(x^n, x^{n-1}y, \frac 12 x^{n-2}y^2, \dots, \frac 1{n!} y^n\right)$. Let
$\widetilde \PPP_n \to \PPP_n$ denote the corresponding universal covering.
\begin{Con}
Let $(X, H)$ be a smooth polarized $n$-dimensional variety. Its space $\Stab_H(X)$ of stability
conditions contains an open subset $\widetilde \PPP_n$, for which skyscraper sheaves of points are stable. 
In the case of abelian varieties, $\widetilde \PPP_n \subset \Stab_H(X)$ is a connected component.
\end{Con}
Such stability conditions could be constructed by an inductive procedure; the $i$-th induction step
would be an auxiliary notion of stability with respect to a weak notion of central charge $Z_i$
depending on $H^n \ch_0, H^{n-1} \ch_1, \dots, H^{n-i} \ch_i$. Semistable objects would have to
satisfy a quadratic inequality $Q_i$ involving $\ch_{i+1}$. The precise form of $Q_i$ would depend
on the parameters of the stability condition; it would always be contained in the defining ideal of
the rational normal curve, and  the kernel of $Z_i$ would be semi-negative definite with
respect to $Q_i$.

One could hope to prove such inequalities for $i < n$ using a second induction by dimension: for
example, an inequality for $\ch_3$ for stable objects on a fourfold would follow from a Mehta-Ramanathan type
restriction theorem, showing that such objects restrict to semistable objects on
threefolds. As a first test case, one should try to prove that a given tilt-stable object on a threefold
restricts to a Bridgeland-stable object on a divisor of sufficiently high degree.

\subsection*{Related work}
As indicated above, the first breakthrough towards constructing stability conditions on threefolds
(without using exceptional collections) is due to Maciocia and Piyaratne, who proved Theorem
\ref{thm:mainabelian} in the case of principally polarized abelian varieties of Picard rank one
in \cite{Dulip-Antony:I, Dulip-Antony:II}. Their method is based on an extensive analysis of the
behavior of tilt-stability with respect to Fourier-Mukai transforms; in addition to
constructing stability conditions, they show their invariance under Fourier-Mukai transforms.

Our approach is very different, as it only uses the existence of the \'etale self-maps given by multiplication with
$m$. Nevertheless, there are some similarities. For example, a crucial step in their
arguments uses restriction to divisors and curves to control a certain cohomology sheaf of the Fourier-Mukai
transform of $E$, see the proof of \cite[Proposition 4.15]{Dulip-Antony:I}; in Section
\ref{sec:abelianBG0} we use restriction of divisors explicitly and to curves implicitly (when we use Theorem
\ref{thm:SLP}) to control global sections of pull-backs of $E$.

As mentioned earlier, it is easy to construct stability conditions on any variety admitting a
complete exceptional collection; however, it is still a delicate problem to relate them to the
construction proposed in \cite{BMT:3folds-BG}. This was done in \cite{BMT:3folds-BG, Macri:P3} for
the case of $\P^3$, and in \cite{Benjamin:quadric} for the case of the quadric in $\P^4$; these are
the only other cases in which Conjecture \ref{con:strongBG} is known.

There is an alternative conjectural approach towards stability conditions on the quintic
hypersurface in $\P^4$ via graded matrix factorizations, proposed by Toda
\cite{Yukinobu:Gepner-matrix-factorizations, Yukinobu:Gepner-point-quintic}. It is more specific,
but would yield a stability condition that is invariant under certain auto-equivalences; it would
also lie outside of our set $\widetilde \PPP$. His approach would require a stronger
Bogomolov-Gieseker inequality already for slope-stable vector bundles, and likely lead to very
interesting consequences for generating functions of Donaldson-Thomas invariants.

Conjecture \ref{con:strongBG} can be specialized to certain slope-stable sheaves, similar to
Corollary \ref{cor:slopestable}; see \cite[Conjecture 7.2.3]{BMT:3folds-BG}. This statement was
proved by Toda for certain Calabi-Yau threefolds, including the quintic hypersurface, in
\cite{Yukinobu:noteonBG3}. Another case of that conjecture implies a certain Castelnuovo-type
inequality between the genus and degree of curves lying on a given threefold; see
\cite{Becca:Castelnuovo} for its relation to bounds obtained via classical methods.

Our results are at least partially consistent with the expectations formulated in
\cite{Polishchuk:Lagrangian}; in particular, semi-homogeneous bundles are examples of the
Lagrangian-invariant objects  considered by Polishchuk, are semistable for our stability
conditions, and their phases behave as predicted.

\subsection*{Plan of the paper} Appendix \ref{app:SupportProperty} may be of independent interest.
We review systematically the relation between support property, quadratic inequalities for
semistable objects and deformations of stability conditions, and their behaviour under
wall-crossing.

Sections \ref{sec:TiltStability} and \ref{sec:classicalBG} and Appendix \ref{app:deformingtilt} review
basic properties of tilt-stabilty, its deformation properties (fixing a small inaccuracy in
\cite{BMT:3folds-BG}), the conjectural inequality proposed in \cite{BMT:3folds-BG} and variants
of the classical Bogomolov-Gieseker inequality satisfies by tilt-stable objects.

In Section \ref{sec:generalizing} we show that a more general form of Conjecture \ref{con:strongBG}
is equivalent to the original conjecture, whereas Section \ref{sec:reductionalphazero} shows
that both conjectures follows from a special limiting case.

This limiting case is proved for abelian threefolds in Section \ref{sec:abelianBG0}; in the
following Section \ref{sec:newstability} we show that this implies the existence of the open subset
$\widetilde \PPP$ of stabilty conditions described above. Section \ref{sec:abelianspace} shows that in the case
of abelian threefolds, $\widetilde \PPP$ is in fact a connected component, and Section \ref{sec:CY} extends
these results to (crepant resolutions) of quotients of abelian threefolds.

\subsection*{Acknowledgments}

The paper benefitted from many useful discussions with Aaron Bertram, Izzet Coskun, Alice Garbagnati, Bert van Geemen, Daniel Huybrechts, Mart\'i
Lahoz, Antony Maciocia, Eric Miles, Rahul Pandharipande, Dulip Piyaratne, Benjamin Schmidt, Yukinobu
Toda, and we would like to thank all of them. The first author is particularly grateful to Dulip
Piyaratne for many hours explaining the details of \cite{Dulip-Antony:I, Dulip-Antony:II}, including
a long session under the disguise of a PhD defense. We are grateful to the referee for a very
careful reading of the manuscript. 
We also would like to thank for their hospitality the Ohio State University, the University of Bonn, and the University of Edinburgh, where parts of this paper were written.

A.B.~ is supported by ERC starting grant no. 337039 ``WallXBirGeom''.
E.M.~ is partially supported by the NSF grants DMS-1160466 and DMS-1302730/DMS-1523496.
P.S.~ is partially supported by the grant FIRB 2012 ``Moduli Spaces and Their Applications'' and by the national research project ``Geometria delle Variet\`a Proiettive'' (PRIN 2010-11).

\subsection*{Update (March 2016)} 
Counterexamples due to Schmidt \cite{Benjamin:counterexample} and Martinez
\cite{Cristian:counterexample} indicate that Conjectures \ref{con:strongBG} and \ref{con:BGgeneral}
need to be modified in the case of a threefold obtained as the blowup at a point of another
threefold; on the other hand,
they have been verified for all Fano threefolds of Picard rank one \cite{Chunyi:Fano3folds}.

\section{Review: tilt-stability and the conjectural BG inequality}
\label{sec:TiltStability}

In this section, we review the notion of tilt-stability for threefolds introduced in
\cite{BMT:3folds-BG}. We then recall the conjectural Bogomolov-Gieseker type inequality for
tilt-stable complexes proposed there; see Conjecture \ref{con:strongBG} below.

\subsection*{Slope-stability}

Let $X$ be a smooth projective complex variety and let $n\geq1$ be its dimension.
Let $\omega \in\mathrm{NS}(X)_\R$ be a real ample divisor class.

For an arbitrary divisor class $B \in \mathrm{NS}(X)_\R$, we will always consider
the twisted Chern character
$\ch^B(E) = e^{-B} \ch(E)$; more explicitly, we have
\begin{equation}\label{eq:explicitB}
\begin{aligned}
& \ch^B_0=\ch_0= \mathrm{rank} \qquad & \ch^B_2=\ch_2-B \ch_1+\frac{B^2}{2} \ch_0 & \\
& \ch^B_1=\ch_1-B\ch_0 \qquad & \ch^B_3=\ch_3-B\ch_2+\frac{B^2}{2} \ch_1-\frac{B^3}{6} \ch_0&.
\end{aligned}
\end{equation}

We define the slope $\mu_{\omega, B}$ of a coherent sheaf $E$ on $X$ by
\[
\mu_{\omega, B}(E) = \begin{cases}
 + \infty, & \text{ if }\ch^B_0(E)=0,\\
\ & \\
\frac{\omega^{n-1} \ch^B_1(E)}{\omega^n \ch^B_0(E)}, & \text{ otherwise.}
 \end{cases}
\]
When $B=0$, we will often write $\mu_\omega$.

\begin{Def}\label{def:SlopeStability}
A coherent sheaf $E$ is slope-(semi)stable (or $\mu_{\omega, B}$-(semi)stable) if, for all non-zero subsheaves $F \into E$, we have
\[
\mu_{\omega, B}(F) < (\le) \mu_{\omega, B}(E/F).
\]
\end{Def}

Observe that if a sheaf is slope-semistable, then it is either torsion-free or torsion.
Harder-Narasimhan filtrations (HN-filtrations, for short) with respect to slope-stability exist in $\Coh(X)$:
given a non-zero sheaf $E\in\Coh(X)$, there is a filtration
\[
0=E_0 \subset E_1 \subset \dots \subset E_m=E
\]
such that: (i) $A_i := E_i/E_{i-1}$ is slope-semistable, and (ii) $\mu_{\omega, B}(A_1) > \dots >
\mu_{\omega, B}(A_m)$.
We set $\mu^+_{\omega, B}(E):=\mu_{\omega, B}(A_1)$ and $\mu^-_{\omega, B}(E):=\mu_{\omega, B}(A_m)$.

\subsection*{The tilted category}
Let $X$ be a smooth projective threefold. As above, let
$\omega, B$ be real divisor classes with $\omega$ ample.
There exists a \emph{torsion pair} $(\TT_{\omega, B}, \FF_{\omega, B})$ in $\Coh(X)$ defined
as follows:
\begin{align*}
\TT_{\omega, B} &= \stv{E \in \Coh(X)}
{\text{any quotient $E \onto G$ satisfies $\mu_{\omega, B}(G) > 0$}}
= \stv{E}{\mu_{\omega, B}^-(E) > 0} \\
\FF_{\omega, B} &= \stv{E \in \Coh(X)}
{\text{any subsheaf $F \into E$ satisfies $\mu_{\omega, B}(F) \le 0$}}
= \stv{E}{\mu_{\omega, B}^+(E) \le 0}
\end{align*}
Equivalently, $\TT_{\omega, B}$ and $\FF_{\omega, B}$ are the extension-closed
subcategories of $\Coh(X)$ generated by slope-stable sheaves of positive and non-positive slope,
respectively.

\begin{Def}\label{def:BB}
We let $\Coh^{\omega, B}(X) \subset \Db(X)$ be the extension-closure
\[
\Coh^{\omega, B}(X) = \langle \TT_{\omega, B}, \FF_{\omega, B}[1] \rangle.
\]
\end{Def}

By the general theory of torsion pairs and tilting \cite{Happel-al:tilting}, $\Coh^{\omega, B}(X)$
is the heart of a bounded t-structure on $\Db(X)$; in particular, it is an abelian category.

\subsection*{Tilt-stability and the main conjecture}

We now define the following slope function, called \emph{tilt}, on the abelian
category $\Coh^{\omega, B}(X)$: for an object $E \in \Coh^{\omega, B}(X)$, its
tilt $\nu_{\omega, B}(E)$ is defined by
\[
\nu_{\omega, B}(E) = \begin{cases}
 + \infty, & \text{ if }\omega^{2} \ch^B_1(E) = 0,\\
\ & \\
\frac{\omega \ch^B_2(E) - \frac 16 \omega^3 \ch^B_0(E)}{\omega^{2} \ch^B_1(E)}, & \text{ otherwise.}
\end{cases}
\end{equation*}
We think of this as induced by the ``reduced'' central charge
\begin{equation} \label{eq:reducedZnew}
\overline{Z}_{\omega, B}(E) = \omega^2 \ch^B_1(E) + i
\omega \left( \ch^B_2(E) - \frac 16 \omega^2 \ch^B_0(E) \right);
\end{equation}
indeed, if $\overline{Z}_{\omega, B}(E) \neq 0$, then the tilt
$\nu_{\omega, B}(E)$ of $E$ agrees with the slope of that complex number; otherwise it is $+\infty$.

\begin{Def} \label{def:tilt-stable}
An object $E \in \Coh^{\omega, B}(X)$ is \emph{tilt-(semi)stable} if, for all non-trivial subobjects
$F \into E$, we have
\[
\nu_{\omega, B}(F) < (\le) \nu_{\omega, B}(E/F).
\]
\end{Def}
Tilt-stability gives a notion of stability, in the sense that Harder-Narasimhan filtrations exist.

The following conjecture is the main topic of \cite{BMT:3folds-BG}:

\begin{Con}[{\cite[Conjecture 1.3.1]{BMT:3folds-BG}}] \label{con:strongBG}
For any $\nu_{\omega,B}$-semistable object $E\in \Coh^{\omega, B}(X)$ satisfying
$\nu_{\omega, B}(E) = 0$, we have the following generalized Bogomolov-Gieseker inequality
\begin{equation} \label{eq:strongBG}
\ch^B_3(E) \le \frac{\omega^2}{18}\ch^B_1(E).
\end{equation}
\end{Con}


\subsection*{Properties of tilt-stability}

We will often fix $B$ and vary $\omega$ along a ray in the ample cone via
\[ \omega = \sqrt{3} \alpha  H \]
for some given integral ample class $H \in \NS(X)$.\footnote{We follow the convention of \cite{Macri:P3} by inserting
a factor of $\sqrt{3}$ above. This ensures that walls of semistability are semicircles, in analogy
to the case of Bridgeland stability on surfaces. In particular, results from \cite{Aaron-Daniele,
Maciocia:walls} carry over more directly.}

To prove that tilt-stability is a well-behaved property, one needs to use variants of the classical
Bogomolov-Gieseker inequality for slope-semistable sheaves; in particular, this leads to the following statements:

\begin{Rem}\label{rmk:OpennessAndHN}
\begin{enumerate}
\item
Tilt-stability is an open property. More precisely, assume that $E\in\Db(X)$ is $\nu_{\omega,
B}$-stable with $\omega = \sqrt{3}\alpha H$.
Then the set of pairs $(\alpha', B') \in \R_{>0} \times \NS(X)_\R$ such that $E$ is
$\nu_{\sqrt{3}\alpha'H, B'}$-stable is open.
\item The boundary of the above subset of $\R_{>0} \times \NS(X)_\R$ where $E \in \Db(X)$ is tilt-stable is given by a locally finite
collection of \emph{walls}, i.e., submanifolds of real codimension one.
\end{enumerate}
\end{Rem}
Unfortunately, a slightly stronger statement was claimed in \cite[Corollary 3.3.3]{BMT:3folds-BG},
but (as noted first by Yukinobu Toda) the proof there only yields the above claims. We will therefore review these statements
in more detail in Section \ref{sec:classicalBG} and Appendix \ref{app:deformingtilt}; one can also deduce them with the same arguments as
in the surface case, treated in detail in \cite[Section 3]{Toda:ExtremalContractions}.

\begin{Rem} \label{rmk:twowalltypes}
It can be helpful to distinguish between two types of walls for tilt-stability, see Proposition
\ref{prop:tiltwallcrossing}. Locally, a wall for
tilt-stability of $E$ is described by the condition
$\nu_{\omega, B}(F) = \nu_{\omega, B}(E)$ for a destabilizing subobject $F$.
This translates into the condition that either
\begin{enumerate}
\item $\overline{Z}_{\omega, B}(F)$ and $\overline{Z}_{\omega, B}(E)$ are linearly dependent, or
that
\item $\nu_{\omega, B}(E) = +\infty$.
\end{enumerate}
\end{Rem}

In the limit $\omega \to +\infty\cdot H$, tilt-stability becomes
closely related to slope-stability:
\begin{Lem} \label{lem:largelimit}
Let $H, B$ be fixed divisor classes with $H$ ample, and let $\omega = \sqrt{3}\alpha H$ for $\alpha
\in \R_{>0}$.
Then
\begin{enumerate}
\item The category $\Coh^{\omega, B}(X)$ is independent of $\alpha$.
\item Moreover, its subcategory of objects $E$ with $\nu_{\omega, B}(E) = +\infty$ is independent of
$\alpha$.
\item \label{enum:limitstable}
If $E \in \Coh^{H, B}(X)$ is $\nu_{\omega, B}$-semistable for
$\alpha \gg 0$, then it satisfies one of the following conditions:
\begin{enumerate}
\item $H^{-1}(E) = 0$ and $H^0(E)$ is a $\mu_{\omega,B}$-semistable torsion-free sheaf.
\item $H^{-1}(E) = 0$ and $H^0(E)$ is a torsion sheaf.
\item $H^{-1}(E)$ is a $\mu_{\omega,B}$-semistable sheaf and $H^0(E)$ is either 0, or supported in
dimension $\le 1$.
\end{enumerate}
Conversely, assume $E \in \Coh(X)$ is a $\mu_{\omega,B}$-\emph{stable} torsion-free sheaf.
\begin{enumerate}
\item If $H^2\ch_1^B(E)>0$, then $E\in\Coh^{H,B}(X)$ and it is $\nu_{\omega, B}$-stable for $\alpha \gg 0$.
\item If $H^2\ch_1^B(E)\leq0$, then $E[1]\in\Coh^{H,B}(X)$; if moreover $E$ is a vector bundle, then it is $\nu_{\omega, B}$-stable for $\alpha \gg 0$.
\end{enumerate}
\end{enumerate}
\end{Lem}
\begin{proof}
The first two statements are immediate to see.
The arguments for part \eqref{enum:limitstable} are completely analogous to the case of Bridgeland stable objects on surfaces, first
treated in \cite[Proposition 14.2]{Bridgeland:K3}; see also \cite[Proposition 7.2.1]{BMT:3folds-BG} for the first part.
\end{proof}

\section{Classical Bogomolov-Gieseker type inequalities}
\label{sec:classicalBG}

In this section, we review a result from \cite{BMT:3folds-BG} that shows that
tilt-stable objects on $X$ satisfy variants of the classical Bogomolov-Gieseker inequality.

We continue to assume that $X$ is a smooth projective threefold. Throughout this section,
 let $H \in \NS(X)$ be a polarization, $\omega = \sqrt 3 \alpha H$ for $\alpha > 0$, and
$B \in \NS(X)_\R$ arbitrary.

First we recall the classical Bogomolov-Gieseker inequality:
\begin{Def}\label{def:discriminant}
The discriminant of $E$ with respect to $H$ is defined by
\[
\Delta_H(E) := H \left(\ch_1(E)^2 - 2 \ch_0(E) \ch_2(E)\right)
 = H \left(\ch_1^B(E)^2 - 2 \ch^B_0(E) \ch^B_2(E)\right).
\]
\end{Def}

\begin{Thm}[Bogomolov, Gieseker] \label{thm:BGclassical}
Assume that $E$ is a $\mu_{H}$-semistable torsion-free sheaf on $X$.
Then $\Delta_H(E) \ge 0$.
\end{Thm}

However, a sheaf $F$ supported on a divisor $D \subset X$ does not necessarily
satisfy $\Delta_H(F) \ge 0$ (even if it is the push-forward of a slope-stable sheaf); indeed, we may
have $H D^2 < 0$.
This leads us to modify the inequality to a form that also holds for
torsion sheaves, and in consequence for tilt-stable objects. We first need the
following easy observation
(see, for example, the proof of \cite[Corollary 7.3.3]{BMT:3folds-BG}):
\begin{Lem}\label{lem:boundeffectiveD2}
There exists a constant $C_H\geq0$ such that for every
effective divisor $D \subset X$, we have
\[
C_H \left(H^2 D\right)^2 + H.D^2 \ge 0.
\]
\end{Lem}
(Note that for abelian threefolds, we may take $C_H = 0$.)

\begin{Def}
We define the $H$-discriminant as the following quadratic form:
\begin{equation} \label{eq:Hdiscriminant}
\overline{\Delta}_H^B := \left(H^2 \ch_1^B\right)^2 - 2 H^3 \ch_0^B H.\ch_2^B.
\end{equation}
For the second definition, choose a rational non-negative constant $C_H$ satisfying the conclusion of
Lemma \ref{lem:boundeffectiveD2}. Then
\begin{equation}
\Delta_{H, B}^C := \Delta_H + C_H \left(H^2 \ch_1^B\right)^2.
\end{equation}
\end{Def}

\begin{Thm}[{\cite[Theorem 7.3.1, Corollaries 7.3.2 and 7.3.3]{BMT:3folds-BG}}] \label{thm:BGvariants}
Let $X$ be a smooth projective threefold with ample polarization $H \in \NS(X)$.
Assume that $E$ is $\nu_{\omega, B}$-semistable for $\omega = \sqrt{3}\alpha H$ and $B \in \NS(X)_\R$.
Then
\[ \overline{\Delta}_H^B(E) \ge 0 \quad \text{and} \quad \Delta_{H, B}^C(E) \ge 0. \]
\end{Thm}
This was proved for rational $B$ in \cite{BMT:3folds-BG};
we will give a self-contained proof of the rational case with a slightly
different presentation below, and extend it to arbitrary $B$ in Appendix \ref{app:deformingtilt}.

We think of $\HDeltaBC$ as the composition
\[
K(X) \xrightarrow{v_H^B} H^0(X, \R) \oplus \NS(X)_\R \oplus \R \xrightarrow{q_H^B} \R
\]
where the first map is given by
\[
v_H^B(E) = \left(\ch_0^B(E), \ch_1^B(E), H \ch_2^B(E)\right)
\]
and where $q_H^B$ is the quadratic form
\begin{equation*}
\left(r, c, d\right) \mapsto H c^2  +C_H \left(H^2 c\right)^2 - 2rd.
\end{equation*}
If $B$ is rational, then the image of $v_H^B$ (and of $\overline{v}_H^B$, defined in Remark
\ref{rmk:DeltaBar} below) is a finite rank lattice.

Notice that $\overline{Z}_{\omega, B}$ as defined in equation \eqref{eq:reducedZnew} factors via $v_H^B$.
Its relation to $q_H^B$ is controlled by
the following immediate consequences of the Hodge index theorem:

\begin{Lem}\label{lem:discriminantsignature}
The quadratic form $q_H^B$ has signature $(2, \rho(X))$.

The kernel of $\overline{Z}_{\omega, B}$ is negative definite with respect to $q_H^B$.
\end{Lem}

This makes our situation analogous to the one in Appendix \ref{app:SupportProperty}; in particular,
Theorem \ref{thm:BGvariants} implies a version of the support property for tilt-stable objects.

\begin{Lem} \label{lem:halfspace}
Let $\nu \in \R \cup \{+\infty\}$.
Then there exists a half-space
\[ \H_{\omega, B, \nu} \subset H^0(X, \R) \oplus \NS(X)_\R \oplus \R \]
of codimension one with the following properties:
\begin{enumerate}
\item For any object $E \in \Coh^{\omega, B}(X)$ with
$\nu_{\omega, B}(E) = \nu$, we have
\[ v_H^B(E) \in \H_{\omega, B, \nu}. \]
\item \label{enum:convexcone}
The intersection of $\H_{\omega, B, \nu}$ with the set defined by $q_H^B(\blank) \ge 0$
is a real convex cone.
\end{enumerate}
\end{Lem}
\begin{proof}
We define $\H_{\omega, B, \nu}$ as the preimage under $\overline{Z}_{\omega, B}$
of the ray in the complex plane that has slope $\nu$, starting at the origin; this ensures the first
claim. The second claim is a general fact about quadratic forms, see Lemma
\ref{lem:convexcone}.
\end{proof}
Note that by definition, a half-space is closed; indeed, we may have $v_H^B(E) = 0$
iff $\nu = +\infty$.

\begin{Rem}\label{rmk:DeltaBar}
If we replace $v_H^B$ with the map
\[
K(X) \xrightarrow{\overline{v}_H^B} \R^3, \quad
\overline{v}_H^B(E) = \left(H^3 \ch_0^B(E), H^2 \ch_1^B(E), H \ch_2^B(E)\right)
\]
and $q_H^B$ with the obvious quadratic form $\overline{q}_H^B$ on $\R^3$, then $\HDeltaB = \overline{q}_H^B \circ \overline{v}_H^B$ and the analogues of Lemma \ref{lem:discriminantsignature} and Lemma \ref{lem:halfspace} hold.
\end{Rem}

\begin{proof}[Proof of Theorem \ref{thm:BGvariants}, case $H^2B \in \Q$]
We prove the statement for $\HDeltaBC$ under the assumption that $H^2B$ is rational. The proof for
$\HDeltaB$ follows similarly due to Remark \ref{rmk:DeltaBar}, and the non-rational case will be
treated in Appendix \ref{app:deformingtilt}.

We proceed by induction on $H^2 \ch_1^B(E)$, which by our assumption is a non-negative function with discrete values on objects of
$\Coh^{H, B}(X)$.

We start increasing $\alpha$.
If $E$ remains stable as $\alpha \to +\infty$, we apply Lemma \ref{lem:largelimit}, \eqref{enum:limitstable};
by Theorem \ref{thm:BGclassical} (for torsion-free
slope-semistable sheaves) and Lemma \ref{lem:boundeffectiveD2} (for torsion sheaves) one easily
verifies that $E$ satisfies the conclusion in any of the possible cases.

Otherwise, $E$ will get destabilized. Note that as $\alpha$ increases, all possible destabilizing subobjects and
quotients have strictly smaller $H^2 \ch_1^B$, which satisfy
the desired inequality by our induction assumption. This is enough to ensure that $E$ satisfies
well-behaved wall-crossing: following the argument of \cite[Proposition 9.3]{Bridgeland:K3}
it is enough to know a support property type statement for all potentially destabilizing classes.

Hence there will be a wall $\alpha = \alpha_W$ where $E$ is strictly
$\nu_{\sqrt{3}\alpha_W H, B}$-semistable; let
\[
0\to E_1 \to E \to E_2 \to 0 \]
be a short exact sequence where both $E_1$ and $E_2$ have the same tilt as $E$.
Then both $E_1$ and $E_2$ have strictly smaller $H^2 \ch_1^B$; so they satisfy the inequality
$\HDeltaBC(E_i) \ge 0$ by the induction assumption. In other words, $v_H^B(E_i)$ are contained in
the cone described in Lemma \ref{lem:halfspace}, \eqref{enum:convexcone}; by convexity, the same
holds for
\[
v_H^B(E) = v_H^B(E_1) + v_H^B(E_2).
\]
\end{proof}

We now turn to some consequences of Theorem \ref{thm:BGvariants}.

\begin{Lem} \label{lem:trivial}
Let $Q$ be a quadratic form of signature $(1, r)$. Let $\CC^+$ be the closure of one of the
two components of the positive cone given by $Q(x) > 0$. Assume that $x_1, \dots, x_m \in \CC^+$,
and let $x := x_1 + \dots, x_m$.
Then
\[ Q(x_i) \le Q(x) \quad \text{for all $i$,} \]
with equality if and only if for all $i$, we have that $x_i$ is proportional to $x$ and $Q(x_i) =
Q(x) = 0$.
\end{Lem}
\begin{proof}
This follows immediately
from the easy fact that if $x, y \in \CC^+ - \{0\}$, then the bilinear form associated to $Q$
satisfies $(x,y) \ge 0$, with equality if and only if $x,y$ are proportional with $Q(x)= Q(y) = 0$.
\end{proof}

\begin{Cor}\label{cor:DeltaJH}
Assume that $E$ is strictly $\nu_{\omega, B}$-semistable with $\nu_{\omega, B}(E) \neq +\infty$.
Let $E_i$ be the Jordan-H\"older factors of $E$. Then
\[ \HDeltaB(E_i) \le \HDeltaB(E) \quad \text{for all $i$.} \]
Equality holds if and only if all $\overline{v}_H^B(E_i)$ are proportional to $\overline{v}_H^B(E)$ and
satisfy $\HDeltaB(E_i) = \HDeltaB(E) = 0$.
In particular, if $E$ is $\nu_{\omega', B'}$-\emph{stable} for some $\omega', B'$ with
$\omega'$ proportional to $\omega$, then the inequality is strict.

The same statements hold with $\HDeltaB$ and $\overline{v}_H^B$ replaced by $\HDeltaBC$ and $v_H^B$,
respectively.
\end{Cor}
The case $\nu = +\infty$ is excluded as in that case we may have
$\overline{v}_H^B(E_i) = 0$ or $\overline{v}_H^B(E_i) = \overline{v}_H^B(E)$.

\begin{proof}
Let $x_i := \overline{v}_H^B(E_i)$ and $x := \overline{v}_H^B(E)$.
By Lemmas \ref{lem:discriminantsignature} and \ref{lem:halfspace}, they satisfy the assumptions of Lemma \ref{lem:trivial}, which then implies
our claim.
\end{proof}

As another application, one obtains the tilt-stability of certain slope-stable sheaves
(see also \cite[Proposition 7.4.1]{BMT:3folds-BG}):

\begin{Cor} \label{cor:Delta0}
\begin{enumerate}
\item Let $F$ be a $\mu_{H,B}$-\emph{stable} vector bundle with $\HDeltaBC(F) = 0$ or $\HDeltaB(F) = 0$. Then
$F$ or $F[1]$ is a $\nu_{\omega, B}$-stable object of $\Coh^{H,B}(X)$.
\item \label{enum:linebundlesstable}
In particular, if $L$ is a line bundle, and if in addition either $c_1(L) - B$ is proportional to
$H$, or we can choose the constant $C_H$ of Lemma \ref{lem:boundeffectiveD2} to be zero, then
$L$ or $L[1]$ is $\nu_{\omega, B}$-stable.
\item \label{enum:tiltstableD0}
Conversely, consider an object
$E \in \Coh^{H, B}(X)$ that is $\nu_{\omega, B}$-stable with
$\HDeltaBC(E) = 0$ or $\HDeltaB(E) = 0$. Then either $E = H^0(E)$ is a $\mu_H$-semistable sheaf,
or $E = H^0(E)$ is supported in dimension $\le 2$,
or $H^{-1}(E) \neq 0$ is a $\mu_H$-semistable sheaf and
$H^0(E)$ has zero-dimensional support.
In addition, $E$ is $\nu_{\omega', B}$-stable for all $\omega'$ proportional to $H$.
\end{enumerate}
\end{Cor}
Note that the choice $C_H = 0$ in particular applies to abelian threefolds (or more generally any
threefold whose group of automorphisms acts transitively on closed points), or to any threefold of Picard rank one.
\begin{proof}
Consider an object $E$ that is $\nu_{\omega, B}$-stable with $\HDeltaB(E) = 0$ or $\HDeltaBC(E) = 0$.
By Corollary \ref{cor:DeltaJH}, $E$ can never become strictly
semistable with respect to $\nu_{\omega', B'}$ as long as $\omega'$ is proportional to $\omega$.
Combined with Lemma \ref{lem:largelimit}, \eqref{enum:limitstable} this implies all our claims.
\end{proof}

The analogue to the case $C_H = 0$ of part \eqref{enum:linebundlesstable} for Bridgeland stability
on surfaces is due to Arcara and Miles, see \cite[Theorem 1.1]{Daniele-Eric:line-bundles}, with a very different proof.

\begin{Prop}\label{prop:Delta0Conjecture}
Assume that $B$ is rational, and let $E \in \Coh^{H, B}(X)$ be a $\nu_{\omega, B}$-stable object with $\HDeltaB(E) = 0$ and $\nu_{\omega,B}(E)=0$.
Then $E$ satisfies Conjecture \ref{con:strongBG}.
\end{Prop}

\begin{proof}
If $F$ is a $\mu_{\omega,B}$-semistable reflexive sheaf on $X$ with $\HDeltaB(F) = 0$, then $F$ is a vector bundle
by \cite[Proposition 3.12]{JasonYogesh:P3},
Further, if $E$ is $\nu_{\omega,B}$-semistable with $\nu_{\omega,B}(E)<+\infty$, then $H^{-1}(E)$ is reflexive
by \cite[Proposition 3.1]{JasonYogesh:P3}.
Hence, the case $H^{-1}(E) \neq 0$ of part \eqref{enum:tiltstableD0} in Corollary \ref{cor:Delta0} can actually be made much more precise: in this case, $H^0(E) = 0$ and $H^{-1}(E)$ is a vector bundle.
In the other case, if $\nu_{\omega,B}(E)=0$, $\HDeltaB(E)=0$, and $H^{-1}(E)=0$, then $H^0(E)$ is a torsion-free sheaf and its double-dual is again locally-free with $\HDeltaB=0$.
In either case, a classical result of Simpson (see \cite[Theorem 2]{Simpson:Higgs} and \cite[Theorem 4.1]{Langer:S-group}) implies that $E$ satisfies Conjecture
\ref{con:strongBG}; see \cite[Proposition 7.4.2]{BMT:3folds-BG}.
\end{proof}

\section{Generalizing the main conjecture}
\label{sec:generalizing}

For this and the following section, we assume that $\omega$ \emph{and} $B$ are proportional to a given ample class
$H \in \NS(X)$:
\begin{equation} \label{eq:omegaBproportionalH}
 \omega = \sqrt{3}\alpha H, \quad B = \beta H.
\end{equation}
We will abuse notation and write
$\ch_i^\beta$ instead of $\ch_i^{\beta H}$, $\Coh^\beta(X)$ instead of $\Coh^{H,\beta H}(X)$, and
$\nu^H_{\alpha, \beta}$ or $\nu_{\alpha, \beta}$  to abbreviate
\[
\nu_{\alpha, \beta} = \sqrt{3} \alpha \nu_{\sqrt{3}\alpha H, \beta H}
= \frac{H \ch_2^\beta - \frac 12 \alpha^2 H^3 \ch_0^\beta}{H^2 \ch_1^\beta}.
\]
We will also write $\HDelta$ instead of $\HDeltaB$, as it is independent of the choice of $\beta$.

The goal of this section is to generalize Conjecture \ref{con:strongBG} to arbitrary
tilt-semistable objects, not just those satisfying $\nu_{\alpha, \beta} = 0$. This generalization
relies on the structure of walls for tilt-stability in $\R_{>0}\times \R$;
it is completely analogous to the case of walls for Bridgeland stability on
surfaces, treated most systematically in \cite{Maciocia:walls}.

\begin{Con} \label{con:BGgeneral}
Let $X$ be a smooth projective threefold, and $H \in \NS(X)$ an ample class.
Assume that $E$ is $\nu^H_{\alpha, \beta}$-semistable. Then
\begin{equation} \label{eq:BGgeneral}
\alpha^2 \HDelta(E) + 4 \left( H \ch_2^\beta(E)\right)^2 - 6 H^2 \ch_1^\beta(E) \ch_3^\beta(E) \ge 0.
\end{equation}
\end{Con}

\begin{Thm} \label{thm:BGgeneral}
Let $X$ be a smooth projective threefold, and $H \in \NS(X)$ an ample class.
Then Conjecture \ref{con:BGgeneral} holds if and only if Conjecture \ref{con:strongBG} holds for
all $\omega, B$ proportional to $H$.
\end{Thm}

We begin with the following aspect of ``Bertram's Nested Wall Theorem'' \cite[Theorem 3.1]{Maciocia:walls}:
\begin{Lem} \label{lem:semicircle}
Assume the situation and notation of Conjecture \ref{con:BGgeneral} with
$\nu_{\alpha, \beta}(E) \neq +\infty$.
Then the object $E$ is
$\nu_{\alpha, \beta}$-semistable along the semicircle $\CC_{\alpha, \beta}(E)$ in the $(\alpha, \beta)$-plane $\R_{>0}\times
\R$ with center $(0, \beta + \nu_{\alpha, \beta}(E))$ and radius
$\sqrt{\alpha^2 + \nu_{\alpha, \beta}(E)^2}$.
\end{Lem}

\begin{proof}
We have to show that $\CC_{\alpha, \beta}(E)$ does not intersect any wall for tilt-stability, which are described in
Remark \ref{rmk:twowalltypes} or
Proposition \ref{prop:tiltwallcrossing}. In our situation, all reduced central charges
$\overline{Z}_{\alpha, \beta}$ factor via the map
\begin{equation}\label{eq:vHbar}
\overline{v}_H \colon K(X) \to \Q^3, \quad w \mapsto \left(H^3 \ch_0(w), H^2 \ch_1(w), H \ch_2(w)\right).
\end{equation}
The first type of wall, case \eqref{enum:standardseqence} in Proposition
\ref{prop:tiltwallcrossing}, can thus equivalently be described as the set of $(\alpha', \beta')$ for
which $\overline{v}_H(F)$ (for some destabilizing subobject $F \into E$) is contained in the
two-dimensional subspace of $\Q^3$ spanned by
$\overline{v}_H(E)$ and the kernel of $\overline{Z}_{\alpha', \beta'}$.

However, this two-dimensional subspace does not vary as
$(\alpha', \beta')$ move within $\CC_{\alpha, \beta}(E)$: the kernel of
$\overline{Z}_{\alpha', \beta'}$ is spanned by $\left(1, \beta', \frac 12 (\alpha'^2 + \beta'^2)\right)$, and the
the vectors
\[ \left(H^3 \ch_0(E), H^2 \ch_1(E), H \ch_2(E)\right), \quad
\left(1, \beta, \frac 12 (\alpha^2 + \beta^2)\right) \quad
\left(1, \beta', \frac 12 (\alpha'^2 + \beta'^2)\right) \]
are linearly dependent if and only if $(\alpha', \beta')$ is contained in $\CC_{\alpha, \beta}(E)$.

In addition, a simple computation shows $H^2 \ch_1^{\beta'}(E) > 0$ for $(\alpha', \beta')  \in\CC_{\alpha, \beta}(E)$;
therefore, the semicircle cannot intersect a wall given by $\nu_{\alpha', \beta'}(E) = +\infty$ either.
\end{proof}

\begin{proof}[Proof of Theorem \ref{thm:BGgeneral}]
We first note that due to Theorem \ref{thm:BGvariants}, Conjecture \ref{con:BGgeneral} holds for all
objects $E$ with $H^2 \ch_1^\beta(E) = 0$. We may therefore assume $\nu_{\alpha, \beta}(E) \neq
+\infty$ throughout the proof.

As an auxiliary step, consider the following statement:
\begin{itemize}
\item[(*)]
Assume that $E$ is $\nu_{\alpha, \beta}$-stable with
$\nu_{\alpha, \beta}(E) \neq +\infty$.
Let $\beta':= \beta + \nu_{\alpha, \beta}(E)$. Then
\begin{equation} \label{eq:BGstrange}
\ch^{\beta'}_3(E) \le \frac 16 \left(\alpha^2 + \nu_{\alpha, \beta}(E)^2\right) H^2\ch_1^{\beta'}(E).
\end{equation}
\end{itemize}
Evidently, Conjecture \ref{con:strongBG} (for the case of $\omega, B$ proportional to $H$) is a
special case of (*). Conversely, consider the assumptions of (*). By Lemma \ref{lem:semicircle},
$E$ is $\nu_{\alpha', \beta'}$-semistable, where $\beta'$ is as above, and
$\alpha'^2 = \alpha ^2 + \nu_{\alpha, \beta}(E)^2$. Moreover, a simple computation shows
$\nu_{\alpha', \beta'}(E) = 0$. Therefore, Conjecture \ref{con:strongBG} implies the statement (*).

Finally, a straightforward computation shows that the inequalities \eqref{eq:BGstrange}
and \eqref{eq:BGgeneral} are equivalent; for this purpose, let us use the abbreviations
$e_i := H^{3-i}\ch_i^\beta(E)$ for $0 \le i \le 3$. Note that by our assumptions, $e_1 > 0$.
With this notation, expanding inequality \eqref{eq:BGstrange} yields:
\[
e_3 - \nu_{\alpha, \beta}e_2 + \frac 12 \nu_{\alpha, \beta}^2 e_1 - \frac 16 \nu_{\alpha, \beta}^3
e_0
 \le \frac 16 \alpha^2 e_1 + \frac 16 \nu_{\alpha, \beta}^2 e_1 - \frac 16 \alpha^2 \nu_{\alpha,
\beta}e_0 - \frac 16 \nu_{\alpha, \beta}^3 e_0
\]
Collecting related terms, substituting
$\nu_{\alpha, \beta} = \frac{-\frac 12 \alpha^2 e_0 + e_2}{e_1}$
and multiplying with $6e_1$ yields:
\[
0 \le -6e_1 e_3 + 3e_2 \left(- \alpha^2 e_0 + 2e_2\right) - 2 \left(-\frac 12 \alpha^2 e_0 + e_2\right)^2
+\alpha^2 e_1^2 - \alpha^2 \left(-\frac 12 \alpha^2 e_0 + e_2\right)e_0
\]
This simplifies to \eqref{eq:BGgeneral}.
\end{proof}

\begin{Ex} \label{ex:slopeCorollary}
Assume that $E$ is a slope-stable sheaf such that $c:=H^2 c_1(E)$ is the minimum positive
integer of the form $H^2 F$ for integral divisor classes $F$;
for example, this is the case when $\NS(X) = \Z \cdot H$ and $c_1(E) = H$.
Then $E$ is $\nu_{\alpha, 0}$-stable for all $\alpha > 0$ by \cite[Lemma
7.2.2]{BMT:3folds-BG}. Hence in that case, Conjecture \ref{con:BGgeneral} claims that
\begin{equation}  \label{ineq:slopestable}
3 c \ch_3(E) \le 2 \left(H \ch_2(E)\right)^2.
\end{equation}
This generalizes \cite[Conjecture 7.2.3]{BMT:3folds-BG}. In particular, let $C \subset X$ be a curve
of genus $g$ and degree $d = H C$; then $E = I_C \otimes \OO(H)$ is supposed to satisfy
\eqref{ineq:slopestable}.
Let $K \in \Z$ such that the canonical divisor class $K_X = K H$. By the Hirzebruch-Riemann-Roch Theorem, we have
\[
1 - g = \chi(\OO_C) = \ch_3(\OO_C) - \frac 12 K d.
\]
Since
\[ \ch(I_C \otimes \OO(H)) = \left(1,  H, \frac 12 H^2 - C, \frac 16 H^3 - d - \ch_3(\OO_C)\right),
\]
the inequality \eqref{ineq:slopestable} specializes to the following Castelnuovo type inequality
between genus and degree of the curve (where $D = H^3$ is the degree of the threefold):
\begin{equation} \label{ineq:Castelnuovo}
g \le \frac{2 d^2}{3D} + \frac{5 + 3K}{6} d + 1
\end{equation}
Even for complete intersection threefolds, this inequality does not follow from existing results;
see \cite[Section 3]{Becca:Castelnuovo} for progress in that direction.
\end{Ex}

\begin{Rem} \label{rem:slopeCorollary}
The inequality \eqref{ineq:slopestable} holds when $X$ is an abelian threefold, or a Calabi-Yau threefold of
abelian type. Moreover, since Conjecture \ref{con:BGgeneral} is equivalent to Conjecture
\ref{con:strongBG}, and since the latter has been verified for $\P^3$ in \cite{BMT:3folds-BG, Macri:P3},
and for the quadric threefold in \cite{Benjamin:quadric}, it also applies in these two cases.

The inequality is new even in the case of $\P^3$: for sheaves of rank three, it is slightly weaker than
classically known results, see \cite[Theorem 4.3]{EinHartshorneVogelaar} and \cite[Theorem
1.2]{Miro-Roig:chernclassesrank3}, but no such results are known for higher rank.
\end{Rem}

\section{Reduction to small $\alpha$}
\label{sec:reductionalphazero}

The goal of this section is to reduce Conjecture \ref{con:BGgeneral} to a more natural inequality,
that can be interpreted as an Euler characteristic in the case of abelian threefolds, and which
considers the limit as $\alpha \to 0$ and $\nu_{\alpha, \beta} \to 0$.

We continue to assume that $X$ is a smooth projective threefold with an ample polarization
$H \in \NS(X)$. To give a slightly better control over the limit $\alpha \to 0$, we will modify the
definition of the reduced central charge of \eqref{eq:reducedZnew} to the following form (which is
equivalent for $\alpha \neq 0$):
\begin{equation} \label{eq:reducedZalphabeta}
\overline{Z}_{\alpha, \beta} = H^2 \ch_1^\beta + i \left(H \ch_2^\beta - \frac 12 \alpha^2 H^3 \ch_0^\beta \right)
\end{equation}
It factors via the map $\overline{v}_H$ of \eqref{eq:vHbar}.
Also, as observed in Remark \ref{rmk:DeltaBar}, the $H$-discriminant can be written as the composition $\HDelta = \overline{q} \circ \overline{v}_H$
where $\overline{q}$ is the quadratic form on $\Q^3$ given by
\[
(r, c, d) \mapsto c^2 - 2rd.
\]

Given any $E \in \Coh^\beta(X)$, we define $\obeta(E)$ as follows:
\begin{equation}\label{eq:betabar}
\obeta(E) :=
\begin{cases}
\frac{H\ch_2(E)}{H^2\ch_1(E)} & \text{if $\ch_0(E) = 0$,} \\
\frac{H^2\ch_1(E) - \sqrt{\overline{\Delta}_H(E)}}{H^3\ch_0(E)}& \text{if $\ch_0(E) \neq 0$.} \end{cases}
\end{equation}
The motivation behind this definition is that $\obeta(E)$ is the limit
of a curve $(\alpha(t), \beta(t)) \in \R_{>0} \times \R$ for which both $\alpha(t) \to 0$
and $\nu_{\alpha(t), \beta(t)}(E) \to 0$; in other words, for which the right-hand-side of
the inequality \eqref{eq:BGstrange} goes to zero: this follows from
\begin{equation} \label{eq:ch2barbeta}
H\ch_2^{\obeta(E)}(E) = 0.
\end{equation}
We also point out that $H^2 \ch_1^{\obeta(E)}(E) > 0$ unless $\HDelta(E) = 0$.

The other motivation for the definition of $\bar\beta$ lies in the following observations,
extending Lemma \ref{lem:discriminantsignature}:
\begin{Lem} \label{lem:limitkernel}
The kernel of $\overline{Z}_{0, \obeta(E)}$ (as a subspace of $\R^3$) is contained
in the quadric $\overline{q} =0$, and the map
$(\alpha, \beta) \to \Ker \overline{Z}_{\alpha, \beta}$ extends to a continuous map
from of $\R_{\ge 0} \times \R$ to the projectivization $\CC^-/\R^*$
of the cone  $\CC^- \subset \R^3$ given by $\overline{q} \le 0$.

Moreover, if $\HDelta(E) > 0$, then the quadratic form $\overline{q}$ is positive semi-definite on
the 2-plane spanned by $\overline{v}_H(E)$ and the kernel of $\overline{Z}_{0, \obeta(E)}$.
\end{Lem}
In other words, the vector $\overline{v}_H(E)$ is contained in the tangent plane to the
quadric $\overline{q} = 0$ at the kernel of $\overline{Z}_{0, \obeta(E)}$; see Figure \ref{fig:walls}.

\begin{Rem} \label{rem:Zviakernel}
The map $(\alpha, \beta) \mapsto \Ker \overline{Z}_{\alpha, \beta}$ gives a
homeomorphism from $\R_{\ge 0} \times \R$ onto its image in the closed unit disc
$\CC^-/\R^*$. This can be a helpful visualization, as a central charge is,
up to the action of $\GL_2(\R)$, determined by its kernel.
\end{Rem}

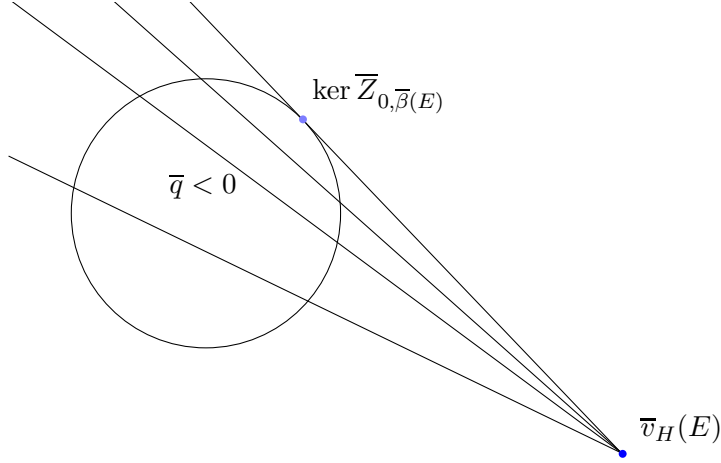
\begin{figure}
\begin{centering}
\definecolor{xdxdff}{rgb}{0.49,0.49,1}
\definecolor{qqqqff}{rgb}{0,0,1}
\begin{tikzpicture}[line cap=round,line join=round,>=triangle 45,x=1.0cm,y=1.0cm]
\clip(-1,-1.5) rectangle (9,5);
\draw(1.62,2.2) circle (1.79cm);
\draw (1,2.86) node[anchor=north west] {$\overline{q}<0$};
\draw (2.9,4.2) node[anchor=north west] {$\ker \overline{Z}_{0,\overline{\beta}(E)}$};
\draw (7.26,-0.34) node[anchor=north west] {$\overline{v}_H(E)$};
\draw [domain=-1.0:7.16] plot(\x,{(-18.07--3.54*\x)/-7.3});
\draw [domain=-1.0:7.16] plot(\x,{(-26.97--4.64*\x)/-6.26});
\draw [domain=-1.0:7.16] plot(\x,{(-26.44--4.38*\x)/-4.92});
\draw [domain=-1.0:7.16] plot(\x,{(-27.58--4.45*\x)/-4.25});
\begin{scriptsize}
\fill [color=qqqqff] (7.16,-1) circle (1.5pt);
\fill [color=xdxdff] (2.91,3.45) circle (1.5pt);
\end{scriptsize}
\end{tikzpicture}\caption{A section of the negative cone $\overline{q}\leq0$ and the tangent plane passing through $\overline{v}_H(E)$ and $\ker \overline{Z}_{0,\overline{\beta}(E)}$.
The other planes through $\overline{v}_H(E)$ intersecting $\overline{q}<0$ correspond to walls of stability for $\overline{v}_H(E)$.}
\label{fig:walls}
\end{centering}
\end{figure}

\begin{proof}[Proof of Lemma \ref{lem:limitkernel}]
The kernel of $\overline{Z}_{\alpha, \beta}$ is spanned by the vector
$\left(1, \beta, \frac 12 (\alpha^2 + \beta^2)\right)$, which has $H$-discriminant
$q_H\left(1, \beta, \frac 12 (\alpha^2 + \beta^2)\right)=  -\alpha^2$. This
proves the first claim.

For the second claim, we just observe that $\left(1, \obeta(E), \frac 12 \obeta(E)^2\right)$ and
$\overline{v}_H(E)$ are orthogonal with respect to the bilinear form on $\R^3$ associated
to $\overline{q}$.
\end{proof}

The following is a limit case of Conjecture \ref{con:BGgeneral}:

\begin{Con} \label{con:BG0}
Let $E \in \Db(X)$ be an object with the following property: there exists an open neighborhood $U
\subset \R^2$
of $(0, \obeta(E))$ such that
for all $(\alpha, \beta) \in U$ with $\alpha > 0$,
either $E$ or $E[1]$ is a $\nu_{\alpha, \beta}$-stable object of
$\Coh^\beta(X)$.
Then
\begin{equation} \label{eq:BG0}
\ch_3^{\obeta(E)}(E) \le 0.
\end{equation}
\end{Con}
Unless $\HDelta(E) = 0$, we can always make $U$ small enough such that $H^2\ch_1^\beta(E) > 0$
for $(\alpha, \beta) \in U$; then $E$ itself is an object of $\Coh^\beta(X)$.

A strengthening of the methods of \cite{Macri:P3} leads to the main result of this section:
\begin{Thm} \label{thm:reduceto0}
Conjectures \ref{con:BG0} and \ref{con:BGgeneral} are equivalent.
\end{Thm}

\begin{Lem} \label{lem:obetanoendpoint}
Let $E \in \Db(X)$ be an object with $\HDelta(E) > 0$ that is $\nu_{\alpha,\beta}$-stable for some
$(\alpha, \beta) \in \R_{>0}\times \R$.
The point $(0, \obeta(E))$ cannot be an endpoint of a wall of tilt-stability for $E$.
Moreover, each of the semicircles of Lemma \ref{lem:semicircle} (along which $E$ has to remain
stable) contains $(0, \obeta(E))$ in its interior.
\end{Lem}
\begin{proof}
Recall the description of walls in Remark \ref{rmk:twowalltypes}. As
$\HDelta(E) > 0$ implies $H^2\ch_1^{\obeta(E)}(E) > 0$, we can exclude the possibility of a wall
given by $\nu_{\omega, B}(E) = +\infty$.
The other type of walls can equivalently be defined by the property that
the kernel of $\overline{Z}_{\alpha, \beta}(E)$  is contained in the 2-plane $\Pi \subset \R^3$ spanned by
$\overline{v}_H(F)$ and $\overline{v}_H(E)$, for some destabilizing subobject $F \into E$.
The signature of $\overline{q}$ restricted to $\Pi$ has to be $(1, 1)$
(as it contains $\overline{v}_H(E)$ and the kernel of $\overline{Z}_{\alpha, \beta}$ for some
$\alpha > 0$).
If $(0, \obeta(E))$ was an endpoint of this wall, then by Lemma \ref{lem:limitkernel} the
kernel of $Z_{0, \obeta(E)}$ would also be contained in $\Pi$; this is a contradiction to the second
assertion of Lemma \ref{lem:limitkernel}.

For the second claim, recall that the semicircles of Lemma \ref{lem:semicircle} do not intersect.
(For example, in Figure \ref{fig:walls}, they are given by the
condition that $\Ker \overline{Z}_{\alpha, \beta}$ is contained in a given plane through
$\overline{v}_H(E)$.) As we shrink the radius of the circles, their center has to converge to
the point with $\alpha = 0$ and $\nu_{\alpha, \beta}(E) = 0$.
\end{proof}

\begin{Lem} \label{lem:Delta0Conjecture}
Objects with $\HDelta(E) = 0$ satisfy both Conjecture \ref{con:BGgeneral} and Conjecture \ref{con:BG0}.
\end{Lem}
\begin{proof}
Proposition \ref{prop:Delta0Conjecture} combined with Theorem \ref{thm:BGgeneral} ensures that such
an object satisfies Conjecture \ref{con:BGgeneral}. If $E$ in addition satisfies the assumptions of
Conjecture \ref{con:BG0}, we consider inequality \eqref{eq:BGgeneral} nearby $(0, \obeta(E))$.
The first term vanishes identically, the second vanishes to second order at $(0, \obeta(E))$.
Therefore, we must have $\ch_3^{\obeta(E)} (E)= 0$; otherwise the third term would only have a simple
zero, in contradiction to Conjecture \ref{con:BGgeneral}.
\end{proof}

\begin{proof}[Proof of Theorem \ref{thm:reduceto0}]
By the previous lemma, we can restrict to the case $\HDelta(E) > 0$ throughout.
First assume that Conjecture \ref{con:BGgeneral} holds. Let $E$ be an object as in the assumptions
of Conjecture \ref{con:BG0} and consider the limit of \eqref{eq:BGgeneral} as $(\alpha, \beta) \to
(0, \obeta)$. Evidently the first term $\alpha^2 \HDelta(E)$ goes to zero; by equation \eqref{eq:ch2barbeta}, the same
holds for the second term $\left(H \ch_2^\beta(E)\right)^2$. Since
$H^2 \ch_1^{\obeta(E)}> 0$, the limit yields exactly \eqref{eq:BG0}.

For the converse, we start with three observations on inequality \eqref{eq:BGgeneral}.
\begin{enumerate}
\item \label{obs1}
Consider
a semicircle given by Lemma \ref{lem:semicircle}. By the proof of Theorem \ref{thm:BGgeneral},
inequality \eqref{eq:BGgeneral} either holds for all points on the semicircle, or it is violated
for all such points; indeed, it is equivalent to inequality \eqref{eq:BGstrange}, which is
just the original Conjecture \ref{con:strongBG} applied at the point where this semicircle
intersects the curve given by $\nu_{\alpha, \beta}(E) = 0$.
\item \label{obs2}
Once we fix $\beta$, it is clear from Theorem \ref{thm:BGvariants} that if \eqref{eq:BGgeneral} holds for a given $\alpha_0$,
then it holds for all $\alpha \ge \alpha_0$.
\item \label{obs3}
Finally, if we consider the semicircles of Lemma
\ref{lem:semicircle} at all points $(\alpha, \beta)$ with $\alpha > 0, \beta = \obeta(E)$,
then by Lemma \ref{lem:obetanoendpoint} they fill up all points of $\R_{>0} \times \R$ with $H^2 \ch_1^\beta(E) > 0$.
\end{enumerate}

Now assume that Conjecture \ref{con:BG0} holds.  We proceed by induction on $\overline{\Delta}_H(E)$
(recall that $\HDelta$ only obtains non-negative integers for tilt-stable objects $E$).

For contradiction, let $E$ be an object that is $\nu_{\alpha, \beta}$-stable, with $\overline{\Delta}_H(E) > 0$, and that
violates conjecture \eqref{eq:BGgeneral} at this point. By Lemma \ref{lem:obetanoendpoint} and
observation \eqref{obs1} above, we may assume $\beta = \obeta(E)$.

Now fix $\beta = \obeta(E)$ and start decreasing $\alpha$.
Since we assume \eqref{eq:BGgeneral} to be violated, we must have $\ch_3^{\obeta(E)}(E) > 0$.
If $E$ were to remain stable as $\alpha \to 0$, then by Lemma \ref{lem:obetanoendpoint} it would be
stable in a neighborhood of $(0, \obeta(E))$ as in the conditions of Conjecture \ref{con:BG0}; this
is a contradiction.

Therefore there must be a point $\alpha_0$ where $E$ is strictly $\nu_{\alpha_0,\obeta(E)}$-semistable; let
$E_i$ be the list of its Jordan-H\"older factors. By observation \eqref{obs2},
$E$ still violates conjecture \eqref{eq:BGgeneral} at $(\alpha_0, \obeta(E))$. On the other hand, by
Corollary \ref{cor:DeltaJH}, $\HDelta(E_i) < \HDelta(E)$ for each $i$; by the induction assumption, $E_i$ satisfies
Conjecture \ref{con:BGgeneral}.

Now the conclusion follows just as in Lemma \ref{lem:blah}:
consider the left-hand-side of \eqref{eq:BGgeneral} as a quadratic form on $\R^4$ with coordinates
$(H^3 \ch_0^\beta, H^2 \ch_1^\beta, H \ch_2^\beta, \ch_3^\beta)$.
The kernel of $\overline{Z}_{\alpha, \beta}$, considered as a subspace of $\R^4$, is
negative semi-definite with respect to the quadratic form. Therefore, the claim follows from Lemma
\ref{lem:convexcone}.
\end{proof}

\section{Tilt stability and \'etale Galois covers}

Consider an \'etale Galois cover $f \colon Y \to X$ with covering group $G$; in
other words, $G$ acts freely on $Y$ with quotient $X = Y/G$. In this section, we will show
that tilt-stability is preserved under pull-back by $f$.

For this section, we again let $\omega, B \in \NS(X)_\R$ be arbitrary classes
with $\omega$ a positive real multiple of an ample.
\begin{Prop}\label{prop:EtaleAndStability}
If $E \in \Db(X)$, then
\begin{enumerate}
\item\label{case:Etale1} $E\in\Coh^{\omega, B}(X)$ if and only if $f^*E\in\Coh^{f^*\omega, f^*B}(Y)$, and
\item\label{case:Etale2} $E$ is $\nu_{\omega,B}$-semistable if and only if $f^*E$ is
$\nu_{f^*\omega,f^*B}$-semistable.
\end{enumerate}
\end{Prop}

\begin{proof}
The pull-back formula for Chern characters immediately gives
\[
\mu_{f^*\omega,f^*B}(f^*F)= \mu_{\omega, B}(F) \quad \text{and} \quad
\nu_{f^*\omega,f^*B}(f^*E)= \nu_{\omega, B}(E).
\]
By \cite[Lemma 3.2.2]{HL:Moduli}, a torsion-free sheaf $F$ is $\mu_{\omega,B}$-semistable if and only
if $f^*F$ is $\mu_{f^*\omega,f^*B}$-semistable, which directly implies \eqref{case:Etale1}.

Now consider $E \in \Coh^{\omega, B}(X)$. Part \eqref{case:Etale1} and the above computation shows
that if $E$ is tilt-unstable, then so is $f^*E$. Conversely, assume that $f^*E$ is
tilt-unstable. Let $F \into f^*E$ be the first step in its Harder-Narasimhan filtration with respect
to $\nu_{f^*\omega, f^*B}$. Since $f^*E$ is $G$-equivariant, and since the HN filtration is unique and
functorial, the object $F$ must also be $G$-equivariant. Hence it is the pull-back of an object $F'$ in
$\Db(X)$. Using part \eqref{case:Etale1} again, we see that $F'$ must be an object of $\Coh^{\omega,
B}(X)$. Applying the same arguments to the quotient $f^*E/F$, we see that $F'$ is a destabilizing
subobject of $E$ in $\Coh^{\omega, B}(X)$.
\end{proof}

\begin{Ex}\label{ex:multiplication}
Let $n\in\Z_{>0}$.
Let $X=Y$ be an abelian threefold and let $\underline{n}\colon X\to X$ be the multiplication by $n$ map.
Then $\un$ has degree $n^6$, and $\un^*H=n^2 H$ for any  class $H\in\mathrm{NS}(X)$; see e.g.
\cite[Corollary 2.3.6 and Chapter 16]{Birkenhake-Lange}.
\end{Ex}

We also obtain directly the following consequence:
\begin{Prop}\label{prop:conjet}
If Conjecture \ref{con:strongBG} holds for tilt-stability
with respect to $\nu_{f^*\omega, f^*B}$ on $Y$, then it also holds for tilt-stability with respect
to $\nu_{\omega, B}$ on $X$.
\end{Prop}

\section{Abelian threefolds}
\label{sec:abelianBG0}

Let $(X, H)$ be a polarized abelian threefold.
In this section we prove Theorem \ref{thm:mainabelian}.

Most of this section will be concerned with proving Conjecture \ref{con:BG0}, the case where $\omega$ and $B$ are proportional to $H$.
For $(\alpha,\beta)\in\R_{>0}\times\R$, we let $\omega=\sqrt{3} \alpha H$ and $B = \beta H$.
We can also assume that $H$ is the class of a \emph{very ample} divisor, which, by abuse of
notation, will also be denoted by $H$.

We let $E\in\Db(X)$ be an object satisfying the assumptions of Conjecture \ref{con:BG0}.
By Lemma \ref{lem:Delta0Conjecture}, we can also assume $\overline{\Delta}_H(E)>0$, and so
$H^2\ch_1^{\obeta(E)}(E)>0$.
We proceed by contradiction, and assume that
\[
\ch_3^{\obeta(E)}(E) >0.
\]

\subsubsection*{Idea of the proof}
Consider the Euler characteristic of the pull-backs
\[
\un^* \left(E(-\obeta(E)H)\right)
\]
via the multiplication by $n$ map.
If we pretend that $E(-\obeta(E)H)$ exists, this Euler characteristic grows proportional to $n^6$;
we will show a contradiction via restriction of sections to divisors.

The proof naturally divides into two cases: if $\obeta(E)$ is rational, then
$\un^* \left(E(-\obeta(E)H)\right)$ exists when $n$ is sufficiently divisible, and the above
approach works verbatim; otherwise, we need to use Diophantine approximation of $\obeta(E)$.

\subsection*{Proof of Conjecture \ref{con:BG0}, rational case}\label{subsec:rational}
We assume that $\overline{\beta}(E)$ is a \emph{rational} number.

\subsubsection*{Reduction to $\overline{\beta}(E)=0$}
Let $q \in \Z_{>0}$ such that $q \obeta(E) \in \Z$, and
consider the multiplication map $\underline{q}:X\to X$.
By Proposition \ref{prop:EtaleAndStability}, $\underline{q}^*E$ still violates Conjecture \ref{con:BG0}.
By definition, we have
\[
\obeta({{\uq}^*E}) = q^2 \obeta(E) \in \Z.
\]
Replacing $E$ with $\uq^*E$, we may assume that $\obeta(E)$ is an integer.
Replacing $E$ again with $E\otimes \OO_X(-\obeta(E) H)$, we may assume that $E$ satisfies the
assumptions of Conjecture \ref{con:BG0}, as well as
\begin{itemize}
\item $\obeta(E)=0$, and so $H.\ch_2(E)=0$, and
\item $\ch_3(E)>0$, and so $\ch_3(E)\geq1$.
\end{itemize}

\subsubsection*{Asymptotic Euler characteristic}
We look at $\chi(\OO_X,\un^*E)$, for $n\to\infty$.
By the Hirzebruch-Riemann-Roch Theorem, we have
\begin{equation}\label{eq:EulerCharacteristic}
\chi(\OO_X,\un^*E) = n^6 \ch_3(E) \geq n^6.
\end{equation}
The goal is to bound $\chi(\OO_X,\un^*E)$ from above with a lower order in $n$.

\subsubsection*{First bound}
We claim that
\begin{equation}\label{eq:FirstBound}
\chi(\OO_X,\un^*E) \leq \hom(\OO_X,\un^*E) + \ext^2(\OO_X,\un^*E).
\end{equation}
Indeed, both $\un^*E$ and $\OO_X[1]$ are objects of $ \Coh^{\beta=0}(X)$.
Hence, for all $k\in\Z_{>0}$, we have
\begin{align*}
&\hom^{-k-1}(\OO_X,\un^*E)=\hom^{-k}(\OO_X[1],\un^*E)=0,\\
&\hom^{k+2}(\OO_X,\un^*E)=\hom^{k+3}(\OO_X[1],\un^*E)=\hom^{-k}(\un^*E,\OO_X[1]) =0.
\end{align*}

\subsubsection*{Hom-vanishing from stability}
To bound the above cohomology groups, we use Hom-vanishing between line bundles and $\un^*E$.
By Corollary \ref{cor:Delta0}, all objects of  $\Coh^\beta(X)$ of the form $\OO_X(uH)$ and $\OO_X(-uH)[1]$ are $\nu_{\alpha,\beta}$-stable, for all $u>0$ and $\beta$ close to 0.
For $(\beta,\alpha)\to(0,0)$, we have
\begin{equation}\label{eq:nu0}
\begin{split}
& \nu_{\alpha,\beta}(\OO_X(uH)) \to \frac u2 >0\\
& \nu_{\alpha,\beta}(\OO_X(-uH)[1]) \to -\frac u2 <0\\
& \nu_{\alpha,\beta}(\un^*E) \to 0,
\end{split}
\end{equation}
and therefore
\[
\nu_{\alpha,\beta}(\OO_X(H)) > \nu_{\alpha,\beta}(\un^*E) > \nu_{\alpha,\beta}(\OO_X(-H)[1]).
\]
Applying the standard Hom-vanishing between stable objects and Serre duality, we conclude
\begin{equation}\label{eq:vanishing}
\Hom(\OO_X(H),\un^*E)=0 \qquad \text{ and }\qquad \Ext^2(\OO_X(-H),\un^*E)=0.
\end{equation}

\subsubsection*{Restriction to divisors}
We will use this Hom-vanishing to restrict sections to divisors; we will repeatedly apply the
following immediate observation:
\begin{Lem} \label{lem:torvanishing}
Let $F_1, \dots, F_m$ be a finite collection of sheaves. Then any globally generated linear
system contains an open subset of divisors $D$ with
\[ \Tor^i(\OO_D, F_j) = 0 \]
for all $i > 0$ and $j = 1, \dots, m$.
\end{Lem}
\begin{proof}
We choose $D$ such that it does not contain any of the associated points of $F_j$, i.e., such that
the natural map $F_j(-D) \to F_j$ is injective.
\end{proof}
In particular, for general $D$, a finite number of short exact sequences restrict
to exact sequences on $D$, and taking cohomology sheaves of a complex $E$
commutes with restriction to $D$.

\subsubsection*{Bound on $\hom(\OO_X,\un^*E)$}
We want to show
\begin{equation}\label{eq:BoundOnHom}
\hom(\OO_X,\un^*E) = O(n^4).
\end{equation}

We consider the exact triangle in $\Db(X)$
\[
\un^*E \otimes \OO_X(-H) \to \un^*E \to (\un^*E)\otimes \OO_D,
\]
where $D$ is a general smooth linear section of $H$.
By \eqref{eq:vanishing}, we have
\[
\hom(\OO_X,\un^*E) \leq \hom(\OO_X,(\un^*E)\otimes \OO_D).
\]
We consider the cohomology sheaves of $E$ and the exact triangle in $\Db(X)$
\[
H^{-1}(E)[1] \to E \to H^0(E).
\]
Since $D$ is general, Lemma \ref{lem:torvanishing} gives
\[
\hom\bigl(\OO_X,(\un^*E)\otimes \OO_D\bigr)\leq
h^0\bigl(D,(\un^*H^{0}(E))|_D\bigr) + h^1\bigl(D,(\un^*H^{-1}(E))|_D\bigr).
\]

The bound \eqref{eq:BoundOnHom} will then follow from Lemma \ref{lem:h0Bound} below.
We first recall a general bound on global sections of sheaves restricted to hyperplane sections,
which is due to Simpson and Le Potier, and can be deduced as a consequence of the Grauert-M\"ulich Theorem:

\begin{Thm}[{\cite[Corollary 3.3.3]{HL:Moduli}}]\label{thm:SLP}
Let $Y$ be a smooth projective complex variety of dimension $n\geq1$ and let $H$ be a very ample divisor on $Y$.
Let $F\in\Coh(Y)$ be a torsion-free sheaf.
Then, for a general sequence of hyperplane section $D_1,\dots,D_n\in |H|$ and for all $d=1,\dots,n$, we have
\[
h^0(Y_d,F |_{Y_d}) \leq \begin{cases}
\frac{\ch_0(F)H^n}{d!} \left(\mu_H^{+}(F) + \frac{\ch_0(F)-1}{2} + d\right)^d & \text{ if } \mu_{H}^{+}(F)\geq 0  \\
0   & \text{ if } \mu_{H}^{+}(F)<0
\end{cases},
\]
where $Y_n=Y$ and $Y_d:=D_1\cap\dots\cap D_{n-d}$.
\end{Thm}

Notice that in the actual statement of \cite[Corollary 3.3.3]{HL:Moduli} there is a factor $H^n$; this is already included in our definition of slope.

\begin{Lem}\label{lem:h0Bound}
Let $Q$ be a sheaf on $X$ and let $L$ be a line bundle.
For all $i=0,1,2$ and for $D$ a smooth very general surface in the linear system $|H|$, we have
\[
h^i(D, (\un^*Q \otimes L) |_D) = O(n^4).
\]
\end{Lem}

\begin{proof}
We assume first that $Q$ is torsion-free.
Notice that the multiplication map $\un$ preserves slope-stability and the rank.
Therefore, by Theorem \ref{thm:SLP}, we have
\begin{align*}
h^0(D,(\un^*Q \otimes L)|_D) & \leq \frac{\ch_0(Q) H^3}{2} \left(\mu_H^{+}(\un^*Q \otimes L) + \frac{\ch_0(Q)-1}{2} + 2\right)^2\\
      & = \frac{\ch_0(Q) H^3}{2} \left(\mu_H^+(Q)\right)^2 n^4 + O(n^3).
\end{align*}
The $h^2$-estimate follows similarly, by using Serre Duality on $D$.
Finally, the Hirzebruch-Riemann-Roch Theorem on $D$ gives
\begin{align*}
h^1 (D,(\un^*Q \otimes L)|_D) & = -\chi(D,(\un^*Q \otimes L)|_D) + h^0(D,(\un^*Q \otimes L)|_D) +
h^2(D,(\un^*Q \otimes L)|_D)\\
& = O(n^4).
\end{align*}
This finishes the proof in the torsion-free case.

For a general sheaf $Q$, we take a resolution
\[
0 \to M \to N \to Q \to 0,
\]
with $N$ locally-free and $M$ torsion-free.
Since $D$ is very general, Lemma \ref{lem:torvanishing} applies, giving
\[
h^i(D,(\un^*Q \otimes L)|_D) \leq h^i(D,(\un^*N \otimes L)|_D) + h^{i+1}(D,(\un^*M \otimes L)|_D).
\]
Hence the result follows from the previous case.
\end{proof}

\subsubsection*{Bound on $\ext^2(\OO_X,\un^*E)$}
This is similar to the previous case.
We consider the exact triangle
\[
\un^*E \to \un^*E \otimes \OO_X(H) \to (\un^*E\otimes \OO_X (H)) \otimes \OO_D.
\]
Again, we apply \eqref{eq:vanishing}, Lemma \ref{lem:torvanishing} and Lemma \ref{lem:h0Bound} and
reach
\begin{equation}\label{eq:BoundOnExt}
\begin{split}
\ext^2(\OO_X,\un^*E) & \leq \ext^1(\OO_X,(\un^*E\otimes \OO_X(H)) \otimes \OO_D)\\
  & \leq h^1(D,(\un^*H^0(E) \otimes \OO_X(H))|_D) + h^2(D,(\un^*H^{-1}(E) \otimes \OO_X(H))|_D)\\
  & = O(n^4).
\end{split}
\end{equation}

\subsubsection*{Conclusion}
By \eqref{eq:EulerCharacteristic}, \eqref{eq:FirstBound}, \eqref{eq:BoundOnHom}, and \eqref{eq:BoundOnExt}, we have
\[
n^6 \leq \chi(\OO_X,\un^*E) = O(n^4),
\]
which gives a contradiction for $n$ sufficiently large.

\subsection*{Proof of Conjecture \ref{con:BG0}, irrational case}
Now assume that $\overline{\beta}(E)\in \R\setminus\Q$ is an irrational number.
As a consequence $\ch_0(E)\neq0$ and, for all $\beta\in\Q$, $H\ch_2^\beta(E)\neq0$.

By assumption, there exists $\epsilon>0$ such that $E$ is  $\nu_{\alpha,\beta}$-stable for all $(\alpha,\beta)$ in
\[
V_\epsilon := \left\{(\alpha,\beta)\in\R_{>0}\times\R\,:\, 0< \alpha<\epsilon,\,  \obeta(E)-\epsilon<\beta <\obeta(E)+\epsilon \right\}.
\]

By the Dirichlet approximation theorem, there exists a sequence $\left\{\beta_n=\frac{p_n}{q_n}\right\}_{n\in\N}$ of rational numbers such that
\begin{equation}\label{eq:Dirichlet}
\left| \obeta(E) - \frac{p_n}{q_n} \right| < \frac{1}{q_n^2}< \epsilon
\end{equation}
for all $n$, and with $q_n \to +\infty$ as $n\to+\infty$.

\subsubsection*{The Euler characteristic}
The function $f(\beta) = \ch_3^\beta(E)$ has derivatives
$f'(\beta) = -H \ch_2^\beta(E)$ and $f''(\beta) = H^2 \ch_1^\beta(E)$;
since $H \ch_2^{\obeta}(E) = 0$ and $H^2 \ch_1^{\obeta}(E) > 0$, the point $\beta = \obeta(E)$ is a
local minimum. Thus, for large $n$, we have
\[
\ch_3^{\beta_n}(E) > \ch_3^{\obeta(E)}(E)>0.
\]
Consider the multiplication map $\underline{q_n}\colon X\to X$.
We let
\[
F_n:= \underline{q_n}^*E \otimes \OO_X(-p_nq_nH).
\]
By Lemma \ref{prop:EtaleAndStability}, $F_n$ is $\nu_{\alpha,0}$-stable, for all $\alpha>0$ sufficiently small.
We have
\begin{equation}\label{eq:EulerCharIrrational}
\chi(\OO_X,F_n) = \ch_3(F_n) = q_n^6 \ch_3^{\beta_n}(E) > q_n^6 \ch_3^{\obeta(E)}(E).
\end{equation}
By \eqref{eq:FirstBound}, it is again enough to bound both $\hom(\OO_X,F_n)$ and $\ext^2(\OO_X,F_n)$ from above.

\subsubsection*{Hom-vanishing}
As $\alpha \to 0$, we have
\[ \nu_{\alpha,0}(F_n) \to q_n^2 \frac{H\ch_2^{\beta_n}(E)}{H^2\ch_1^{\beta_n}(E)}. \]
We can bound this term as follows:
\[
\begin{split}
\abs{q_n^2 \frac{H\ch_2^{\beta_n}(E)}{H^2\ch_1^{\beta_n}(E)}}
&=
q_n^2 \abs{ \frac{
 		H\ch_2^{\obeta}(E) - (\beta_n-\obeta) H^2\ch_1^{\obeta}(E) + \frac 12 (\beta_n - \obeta)^2 H^3 \ch_0(E)}
	{H^2\ch_1^{\obeta}(E) - (\beta_n - \obeta) H^3 \ch_0(E)}} \\
&= q_n^2 \abs{ \beta_n - \obeta}
 \abs{ \frac{H^2 \ch_1^{\obeta}(E) - \frac 12 (\beta_n - \obeta)H^3 \ch_0(E)}
       {H^2 \ch_1^{\obeta}(E) - (\beta_n - \obeta)H^3 \ch_0(E)}} \\
& \le 1 \cdot \left( 1 +  \abs{\frac {\frac 12 (\beta_n - \obeta)H^3 \ch_0(E)}
       {H^2 \ch_1^{\obeta}(E) - (\beta_n - \obeta)H^3 \ch_0(E)}} \right) \to 1.
\end{split}
\]
Here we used $H^2 \ch_2^{\obeta(E)}(E) = 0$ in the second equality, and $H^2 \ch_1^{\obeta(E)}(E) >
0$ in the limit.

By comparison with \eqref{eq:nu0}, it follows that
\[
\nu_{\alpha, \beta_n} (\OO_X(3H)) > \nu_{\alpha,\beta_n}(F_n) > \nu_{\alpha,\beta}(\OO_X(-3H)[1])
\]
for $\alpha \to 0$ and $n$ sufficiently large; therefore
\begin{equation}\label{eq:vanishingIrrational}
\Hom(\OO_X(3H),F_n)=0 \qquad \text{ and }\qquad \Ext^2(\OO_X(-3H),F_n)=0.
\end{equation}

\subsubsection*{Bound on $\hom(\OO_X,\um^*F_n)$ and conclusion}
Proceeding as in the rational case, we consider the exact triangle
\[
F_n\otimes \OO_X(-3H) \to F_n \to F_n\otimes \OO_D,
\]
where $D$ is a general smooth surface in the linear system $|3H|$.
By \eqref{eq:vanishingIrrational}, we have
\begin{align*}
\hom(\OO_X,F_n) & \leq \hom(\OO_X,F_n\otimes \OO_D)\\
 & \leq h^0(D,H^0(F_n)|_D) + h^1(D,H^{-1}(F_n)|_D).
\end{align*}
The following is the analogue of Lemma \ref{lem:h0Bound}:

\begin{Lem}\label{lem:h0BoundIrrational}
Let $Q$ be a sheaf on $X$ and let $L$ be a line bundle.
Then
\[
h^i(D, (\underline{q_n}^* Q(-p_nq_n H) \otimes L)|_D) = O(q_n^4),
\]
for all $i$, and for $D$ a general smooth surface in $|3H|$.
\end{Lem}

\begin{proof}
By the same argument as in the proof of Lemma \ref{lem:h0Bound}, we may assume that $Q$ is
torsion-free. Applying Theorem \ref{thm:SLP} in our case we obtain, for general $D$,
\begin{eqnarray*}
\lefteqn{h^0\left(D,(\underline{q_n}^* Q(-p_nq_n H) \otimes L)|_D\right)} \\
& \leq & \frac{\ch_0(Q)(3H)^3}{2} \left(\mu_{3H}^{+}(\underline{q_n}^* Q(-p_nq_n H) \otimes L) + \frac{\ch_0(Q)-1}{2} + 2\right)^2\\
      & = & \frac{3 \ch_0(Q) H^3}{2} \left(\mu_{H,\beta_n}^+(Q)\right)^2 q_n^4 + O(q_n^3)\\
      & =  &\frac{3 \ch_0(Q) H^3}{2} \left(\mu_{H,\obeta(E)}^+(Q)\right)^2 q_n^4 + O(q_n^3),
\end{eqnarray*}
The $h^1$ and $h^2$ bounds follow from Serre duality and the Riemann-Roch Theorem.
\end{proof}

Applying Lemma \ref{lem:h0BoundIrrational} to the cohomology sheaves of $E$ in combination with
Lemma \ref{lem:torvanishing}, we get
\[
\hom(\OO_X,F_n) = O(q_n^4).
\]
The same argument gives a similar bound on $\ext^2(\OO_X,F_n)$ and a contradiction to \eqref{eq:EulerCharIrrational}.
This completes the proof of Conjecture \ref{con:BG0}, and therefore Conjecture \ref{con:BGgeneral}, for abelian threefolds.

\subsection*{Proof of Theorem \ref{thm:mainabelian}}

Let now $B\in\mathrm{NS}(X)_\R$ be an arbitrary divisor class and $\omega$ a positive multiple of $H$.
In the abelian threefold case, we can use Conjecture \ref{con:BG0} to deduce Conjecture
\ref{con:strongBG} in this more general case.

We let $E\in\Coh^{\omega,B}(X)$ be as in Conjecture \ref{con:strongBG}.
We first assume that $B\in\mathrm{NS}(X)_\Q$ is \emph{rational}.
Then, by Proposition \ref{prop:EtaleAndStability}, we can assume $B$ integral.
By taking the tensor product with $\OO_X(-B)$, we can then assume $E$ is $\nu_{\omega,0}$-semistable.
Conjecture \ref{con:strongBG} then follows directly from Conjecture \ref{con:BGgeneral} and Theorem \ref{thm:reduceto0}.

Finally, we take $B$ irrational.
Since \eqref{eq:strongBG} is additive, by considering its Jordan-H\"older factors we can assume $E$ is $\nu_{\omega,B}$-stable.
By using Theorem \ref{thm:BGvariants} and Remark \ref{rmk:OpennessAndHN}, we can deform $(\omega, B)$
to $(\omega', B')$ with $B'$ rational (and $\omega'$ still proportional to $H$), such that  $E$ is still $\nu_{\omega',B'}$-stable with
$\nu_{\omega', B'}(E) = 0$.
But, if \eqref{eq:strongBG} does not hold for $(\omega, B)$, then it does not hold for
$(\omega', B')$ sufficiently close, giving a contradiction to what we just proved.

\section{Construction of Bridgeland stability conditions} \label{sec:newstability}

It was already established in \cite{BMT:3folds-BG} that Conjecture \ref{con:strongBG} implies the
existence of Bridgeland stability conditions on $X$, except that the notion of support property was
ignored. This property ensures that stability conditions deform freely, and exhibit well-behaved
wall-crossing.

In this section, we show that the equivalent Conjecture \ref{con:BGgeneral} is in fact strong enough
to deduce the support property, and to construct an explicit open subset of the space of stability
conditions. In the following section, we will show that in the case of abelian threefolds, this open
set is in fact an entire component of the space of stability conditions.

\subsection*{Statement of results}
Fix a threefold $X$ with polarization $H$;
we assume throughout this section that Conjecture \ref{con:BGgeneral} is satisfied for the pair
$(X, H)$.  We consider the lattice $\Lambda_H\cong\Z^4$ generated by vectors of the form
\[
\left(H^3 \ch_0(E), H^2\ch_1(E), H\ch_2(E), \ch_3(E) \right) \in \Q^4
\]
together with the obvious map $v_H\colon K(X) \to \Lambda_H$.

We refer to Appendix \ref{app:SupportProperty} for the definition of stability conditions on
$\Db(X)$ with respect to $(\Lambda_H, v_H)$; it is given by a pair $\sigma = (Z, \PP)$, where $\PP$
is a \emph{slicing}, and the \emph{central charge} $Z$ is a linear map $Z \colon \Lambda_H \to \C$.
 The main result of \cite{Bridgeland:Stab} shows that the space $\Stab_H(X)$ of such stability
conditions is a four-dimensional complex manifold such that
\[
\ZZ \colon \Stab_H(X) \to \Hom(\Lambda_H, \C), \quad (Z, \PP) \mapsto Z
\]
is a local isomorphism. In Proposition \ref{prop:EffectiveDefo} we make this deformation result more
effective. This result will be essential in the following, where we will construct an  explicit open
subset of this manifold.
We let $\CCC \subset \Lambda_H \otimes \R \cong \R^4$ be the cone over the twisted cubic
\[ \CCC = \stv{\left(x^3, x^2y, \frac 12 xy^2, \frac 16 y^3\right)}{x, y \in \R},\]
which contains $v_H(\OO_X(uH))$ for all $u \in \Z$.

\begin{Def} \label{def:PPP}
Consider the open subset $\VVV \subset \Hom(\Lambda_H, \C)$ of central charges whose kernel intersects
$\CCC$ only at the origin. We let $\PPP \subset \VVV$ be the connected component containing
$Z_H^{\mathrm{basic}}$ defined by
\begin{equation} \label{eq:Zbasic}
Z_H^{\mathrm{basic}}(E) = \left[ - \ch_3(E) + \frac 12 H^2 \ch_1(E) \right] + i \left[H \ch_2(E) - \frac 16 H^3
\ch_0(E) \right].
\end{equation}
Let $\widetilde \PPP$ be its universal covering.
\end{Def}

The goal of this section is the following precise version of Theorem \ref{maintheorem}:
\begin{Thm} \label{thm:mainPPP}
Let $(X, H)$ be a polarized threefold for which Conjecture \ref{con:BGgeneral} is satisfied. Then
there is an open embedding $\widetilde\PPP \subset \Stab_H(X)$ for which the following diagram
commutes:
\[ \xymatrix{
{\widetilde \PPP} \ar@{^{(}->}[r] \ar[d] & \Stab_H(X) \ar[d]^\ZZ \\
{\PPP} \ar@{^{(}->}[r] & \Hom(\Lambda_H, \C)
} \]
\end{Thm}
We will prove this theorem by constructing an explicit family of stability conditions following the
construction of \cite{BMT:3folds-BG}, and then applying the deformation arguments of Proposition
\ref{prop:EffectiveDefo}.

\subsection*{Alternative description of $\PPP$}
We will need a more explicit description of the set $\PPP$ before proceeding to prove our main
result.

The group $\GL_2^+(\R)$ of $2 \times 2$-matrices with positive determinant acts on $\PPP$ on the
left by post-composing a central charge with the induced $\R$-linear map of $\R^2 \cong \C$. There
is also an action of $\R$ on $\PPP$ on the right: for $\beta \in \R$, the multiplication by
$e^{-\beta H}$ in $K(\Db(X))$ corresponds to a linear selfmap of $\Lambda_H \otimes \R$ which leaves
$\CCC$ invariant; therefore we can act on $\PPP$ by pre-composing with this linear map.

\begin{Lem}  \label{lem:skewtocubic}
There is a slice of $\PPP$ with respect to the $\GL_2^+(\R)$-action given by central charges of the
form
\[
Z_{\alpha, \beta}^{a, b} := \left[-\ch^\beta_3 + b H \ch^\beta_2 + a H^2 \ch^\beta_1 \right] + i
\left[H \ch^\beta_2 - \frac 12 \alpha^2 H^3 \ch^\beta_0 \right]
\]
for all  $\alpha, \beta, a, b \in \R$ satisfying $\alpha > 0$ and
\begin{equation} \label{ineq:a}
a > \frac 16 \alpha^2 + \frac 12 \abs{b} \alpha.
\end{equation}
This slice is simply-connected.
\end{Lem}
It follows that it is simultaneously a slice of the $\widetilde \GL_2^+(\R)$-action on
$\widetilde \PPP$.

\begin{proof}
Consider a central charge $Z \in \PPP$. Since $Z(0, 0, 0, 1) \neq 0$ by definition of $\PPP$, we may use the
action of rotations and dilations to normalize to the assumption $Z(0, 0, 0, 1) = -1$. Now consider
the functions
\[
r(x) := \Re Z\left(1, x, \frac 12 x^2, \frac 16 x^3\right) = - \frac 16 x^3 + O(x^2)
\; \text{and} \;
i(x) := \Im Z\left(1, x, \frac 12 x^2, \frac 16 x^3\right) = O(x^2)
\]
for $Z \in \PPP$ normalized as above; their coefficients vary continuously
with $Z$. They can never vanish simultaneously, by definition
of $\PPP$. In the case of $Z_H^{\mathrm{basic}}$, the function
$r(x) = -\frac 16 x^3 + \frac 12 x$ has zeros as
$x = - \sqrt{3}$, $x = 0$, $x = \sqrt{3}$, whereas $i(x) = \frac 12 x^2 - \frac 16$ has zeros at $x = \pm \sqrt{\frac 13}$.
This configuration of zeros on the real line will remain unchanged as $Z$ varies:
$r(x)$ will always have three zeros, and $i(x)$ will have two zeros lying between the first and
second, and the second and third zero of $r(x)$, respectively.

We now use the action of $\R$ on $\PPP$ from the right to ensure that $x = 0$ is always the midpoint of the
two zeros of $i(x)$. The sign of the leading coefficient of $i(x)$ must remain constant as $Z$
varies; therefore, we can use vertical rescaling of $\R^2$ to normalize it to be $+\frac 12$.
Since the sign of $i(0) = \Im Z(\OO_X)$ is constant within this slice, it has to be negative;
hence there exists a unique $\alpha \in \R_{>0}$ such that $i(0) = - \frac 12 \alpha^2$.

On the slice we have constructed thus far, we still have the action of $\R$ given by sheerings of
$\R^2 \cong \C$ that leave the real line fixed. Since $\Im Z(\OO_X) = i(0) < 0$, there is a unique
such sheering that forces $Z(\OO_X)$ to be real. Summarizing, we have constructed a slice
in which all central charges are of the form
\[
Z_{\alpha, \beta = 0}^{a, b} := \left[-\ch_3 + b H \ch_2 + a H^2 \ch_1 \right] + i
\left[H \ch_2 - \frac 12 \alpha^2 H^3 \ch_0 \right].
\]
In this form, the zeros of $i(x) = \frac 12 x^2 - \frac 12\alpha^2$ are $x = \pm\alpha$; thus the kernel of $Z$
intersects the twisted cubics if and only if
\[
a = \frac 16 \alpha^2 \pm \frac 12 b \alpha.
\]
In the case of $Z_H^{\mathrm{basic}}$, we have $\alpha = \sqrt{\frac 13}$, $b = 0$ and $a = \frac
12$, which is bigger than the right-hand-side. It follows that the inequality \eqref{ineq:a} holds
in the whole connected component of our slice.

Conversely, given a central charge $Z_{\alpha, \beta}^{a, b}$ as described in the lemma, we can
first use the action of $\R$ to reduce to the case $\beta = 0$. The coefficients of the linear
functions $\Im Z, \Re Z$ are in one-to-one correspondence with the coefficients of $r(x)$ and
$i(x)$, respectively; these are, up to scaling, uniquely determined by the configurations of zeros
of $r(x)$ and $i(x)$ on the real line. But our conditions ensure that we can continuously deform the
configuration of zeros into the one corresponding to $Z_H^{\mathrm{basic}}$.
\end{proof}

\begin{Rem}
From the proof of the lemma one can also deduce the following more intrinsic description of the set
$\PPP$. Consider the twisted cubic $\overline{\CCC}$ in projective space $\P^3(\R)$. There is an open subset of
central charges $Z$ with the following properties: the hyperplanes $\Im Z = 0$ and $\Re Z = 0$
both intersect $\overline{\CCC}$ in three distinct points; moreover, their configuration on
$\overline{\CCC} \cong S^1$ are such that the zeros of the two functions alternate.
This open set has two components: one of them is $\PPP$, the other is obtained from $\PPP$ by
composing central charges with complex conjugation.

Moreover, one can also deduce the description given in the introduction.
\end{Rem}

Recall the $H$-discriminant
\[
\HDelta = \left(H^2 \ch_1^\beta\right)^2 - 2 H^3 \ch_0^\beta H \ch^\beta_2.
\]
defined in \eqref{eq:Hdiscriminant},
Let us also introduce a notation of the remainder term of \eqref{eq:BGgeneral}:
\[
\HNablab := 4 \left( H \ch_2^\beta(E)\right)^2 - 6 H^2 \ch_1^\beta(E) \ch_3^\beta(E).
\]

\begin{Lem} \label{lem:negativedefinite}
There is an open interval $I_{\alpha}^{a, b} \subset \R_{>0}$ such that the
kernel of $Z_{\alpha, \beta}^{a, b}$ is negative definite
with respect to the quadratic form
$K \HDelta + \HNablab$ for all $K \in I_{\alpha}^{a, b}$.  In case $b = 0$, the interval is given by
$I_{\alpha}^{a, b} = (\alpha^2 , 6a)$. In case $b \neq 0$, it is a subinterval of $(\alpha^2, 6a)$
satisfying $\frac 12 \left(\alpha^2 + 6a\right) \in I_{\alpha}^{a, b}$ for all $b$, and
\[ I_{\alpha}^{a, b'} \subset I_{\alpha}^{a, b} \]
whenever $\abs{b'} > \abs{b}$.
\end{Lem}
\begin{proof}
Let us use the coordinates $e_i := H^{3-i}\ch_i^\beta$ on $\Lambda_H \otimes \R$.
In these coordinates, the kernel of $Z_{\alpha, \beta}^{a, b}$ is generated by the vectors
$(1, 0, \frac 12 \alpha^2, \frac 12 b\alpha^2)$ and $(0, 1, 0, a)$. The intersection matrix for the
symmetric pairing associated to $K \HDelta + \HNablab$ is
\[
\begin{pmatrix}
- K \alpha^2 + \alpha^4 & -3 b \alpha^2 \\
- 3b \alpha^2 & K - 6a
\end{pmatrix}.
\]
The diagonal entries are negative for $K \in (\alpha^2, 6a)$ (which is non-empty by the assumptions
on $a$). In case $b  \neq 0$, we additionally need to ensure that the determinant
\[
\alpha^2 \left(-K^2 + 6a K + K\alpha^2 - 6a \alpha^2 - 9 b^2 \alpha^2\right)
\]
is positive. Solving the quadratic equation, one obtains a subinterval of $(\alpha^2, 6a)$
symmetric around the midpoint $K = \frac 12 \left(\alpha^2 + 6a\right)$ with the properties as
claimed.
\end{proof}

\subsection*{Review: construction of stability conditions}
We will use \cite[Proposition 5.3]{Bridgeland:Stab} to construct stability conditions. It says that
a stability condition is equivalently determined by a pair $\sigma=(Z,\AA)$, where
$Z\colon \Lambda_H\to\C$ is a group homomorphism (called \emph{central charge}) and $\AA\subset\Db(X)$
is the \emph{heart of a bounded t-structure}, which have to satisfy the following three properties:
\begin{enumerate}
\item \label{item:positivity} For any $0 \neq E\in\AA$ the central charge $Z(v_H(E))$ lies in the following semi-closed
upper half-plane:
\begin{equation} \label{eq:Zpositivity}
Z(v_H(E)) \in  \R_{>0} \cdot e^{(0,1]\cdot i\pi}
\end{equation}
\suspend{enumerate}
We can use $\Re Z$ and $\Im Z$ to define a notion of slope-stability
on the abelian category $\AA$ via the slope function $\lambda_\sigma(E)= -\frac{\Re Z(v_H(E))}{\Im
Z(v_H(E))}$
\resume{enumerate}
\item \label{item:HN} With this notion of slope-stability, every object in $E \in \AA$ has a Harder-Narasimhan
filtration $0 = E_0 \into E_1 \into \dots \into E_n = E$ such that each $E_i/E_{i-1}$ is
$\lambda_\sigma$-semistable,
with $\lambda_\sigma(E_1/E_0) > \lambda_\sigma(E_2/E_1) > \dots > \lambda_\sigma(E_n/E_{n-1})$.
\item \label{item:support} (\emph{support property}) There is a constant $C>0$ such that, for all
$\lambda_\sigma$-semistable object $E\in \AA$, we have
\begin{align*}
\lVert v_H(E) \rVert \le C \lvert Z(v_H(E)) \rvert,
\end{align*}
where $\lVert \blank \rVert$ is a fixed norm on $\Lambda_H\otimes\R \cong \R^4$.
\end{enumerate}
For brevity, we will write $Z(E)$ instead of $Z(v_H(E))$. Shifts of $\lambda_\sigma$-(semi)stable
objects are called $\sigma$-(semi)stable.

\subsection*{Explicit construction of stability conditions}
We start by reviewing (a slightly generalized version of) the construction of stability conditions
in \cite{BMT:3folds-BG}.

We define a heart $\AA^{\alpha,\beta}(X)\subset \Db(X)$ as a tilt of
$\Coh^{\beta}(X)$: we let
\begin{align*}
\TT_{\alpha, \beta}' &= \stv{E \in \Coh^{\beta} (X)}
{\text{any quotient $E \onto G$ satisfies $\nu_{\alpha,\beta}(G) > 0$}} \\
\FF_{\alpha,\beta}' &= \stv{E \in \Coh^{\beta} (X)}
{\text{any subobject $F \into E$ satisfies $\nu_{\alpha,\beta}(F) \le 0$}}
\end{align*}
and define
\[
\AA^{\alpha,\beta}(X) = \langle \TT_{\alpha,\beta}', \FF_{\alpha,\beta}'[1] \rangle.
\]

\begin{Thm}[{\cite{BMT:3folds-BG}}] \label{thm:rationalconstruction}
Let $(X, H)$ be a polarized threefold for which Conjecture
\ref{con:BGgeneral} holds.
Assume that $\alpha, \beta \in \Q$, and that $\alpha, \beta, a, b$ satisfy \eqref{ineq:a}. Then
the pair $\sigma = \left(Z_{\alpha, \beta}^{a, b}, \AA^{\alpha, \beta}(X)\right)$ satisfy conditions
\eqref{item:positivity} and \eqref{item:HN} above.
\end{Thm}
\begin{proof}
The case $b = 0$ is \cite[Corollary 5.2.4]{BMT:3folds-BG}, and the same arguments apply here; let us
review them briefly.

The construction of the heart directly ensures that if $E \in \AA^{\alpha, \beta}(X)$, then
$\Im Z_{\alpha, \beta}^{a, b} \ge 0$.
Moreover, if $E\in \AA^{\alpha,\beta}(X)$ is such that $\Im Z_{\alpha,\beta,s}(E)=0$,
then $E$ fits into an exact triangle $F[1] \to E \to T$ where
\begin{itemize}
\item $T$ is a zero-dimensional torsion sheaf, and
\item $F\in\Coh^\beta(X)$ is $\nu_{\alpha,\beta}$-semistable with  $\nu_{\alpha,\beta}(E)=0$ (in
particular, $H^2\ch_1^{\beta}(F)> 0$).
\end{itemize}
We have $Z_{\alpha, \beta}^{a, b}(T) = - \mathrm{length}(T) < 0$ if $T$ is non-trivial. To treat $F[1]$,
observer that $\nu_{\alpha, \beta}(F) = 0$ implies
\[
\frac 12 \alpha^2 H^3 \ch_0^\beta(F) = H \ch_2^\beta(F).
\]
Therefore we can use Conjecture \ref{con:BGgeneral} and Theorem \ref{thm:BGvariants} to estimate
\[
\begin{split}
Z_{\alpha, \beta}^{a, b}(F[1]) &=
\ch_3^\beta(F) - b H\ch_2^\beta(F) - a H^2 \ch_1(F) \\
&\le \frac 16 \alpha^2 H^2 \ch_1(F) + \abs{b} \frac 12 \alpha^2 H^2 \ch_1^\beta(F) - a H^2 \ch_1(F) < 0.
\end{split}
\]

By \cite[Proposition 5.2.2]{BMT:3folds-BG}, the category $\AA^{\alpha,\beta}(X)$ is noetherian.
Since $\Im Z_{\alpha,\beta,s}$ is a discrete subset of $\R$, we can apply \cite[Proposition
B.2]{localP2} to deduce the existence of Harder-Narasimhan filtrations.
\end{proof}

\subsection*{Support property}
The next step towards proving Theorem \ref{thm:mainPPP} is to establish the support property for the
stability conditions constructed in Theorem \ref{thm:rationalconstruction}. Our overall goal
is the following analogue of Theorem \ref{thm:BGvariants}.

Let $\sigma = (Z, \AA) \in \widetilde{\PPP} \subset \Stab_H(X)$ be a stability condition
in the open subset given in Theorem \ref{thm:rationalconstruction}. We may assume that
$Z = Z_{\alpha, \beta}^{a, b}$ is of the form given in Lemma \ref{lem:skewtocubic}. We also
choose a constant $K \in I_{\alpha}^{a, b}$ in accordance with Lemma \ref{lem:negativedefinite}.
\begin{Thm} \label{thm:BGinDim3}
Under the assumptions above, every $\sigma$-semistable object $E$ satisfies
\begin{equation} \label{eq:newBGdim3}
Q_K^\beta(E) : = K \HDelta(E) + \HNablab(E) \ge 0.
\end{equation}
Moreover, up to shift the heart $\AA$ is of the form $\AA = \AA^{\alpha, \beta}(X)$.
\end{Thm}
We will treat only the case $b = 0$; then $I_{\alpha}^{a, b} = (\alpha^2, 6a)$. We
will also shorten notation and write $Z_{\alpha, \beta}^a$ instead of
$Z_{\alpha, \beta}^{a, 0}$, and $I_{\alpha}^a$ instead of $I_{\alpha}^{a,0}$.
The case $b\neq0$ will then follow directly by Proposition \ref{prop:EffectiveDefo}.

The analogy between Theorem \ref{thm:BGinDim3} and Theorem \ref{thm:BGvariants} is reflected also in their proof.
We first treat the rational case:

\begin{Lem}\label{lem:RationalBGinDim3}
Let $(X,H)$ be a polarized threefold and $(\alpha,\beta)\in\Q_{>0}\times\Q$.
Assume that Conjecture \ref{con:BGgeneral} holds for this pair $(\alpha,\beta)$.
Then for any $a > \frac 16 \alpha^2$,
the pair $\sigma_{\alpha,\beta}^a=(Z_{\alpha,\beta}^a,\AA^{\alpha,\beta}(X))$ satisfies the support
property; more precisely, the inequality \eqref{eq:newBGdim3} holds for all
$\sigma_{\alpha,\beta}^a$-semistable objects $E$ and all $K \in I_{\alpha}^a$.
\end{Lem}
We first need an analogue of Lemma \ref{lem:largelimit}.

Let us denote by $H_\beta^i$ the $i$-th cohomology object with respect to the t-structure $\Coh^{\beta}(X)$.

\begin{Lem}\label{lem:LVL}
Let $E\in\AA^{\alpha,\beta}(X)$ be a $\sigma_{\alpha,\beta}^a$-semistable object, for all $a \gg 1$
sufficiently big.  Then it satisfies one of the following conditions:
\begin{enumerate}
\item $H^{-1}_\beta(E)=0$ and $H^0_\beta(E)$ is $\nu_{\alpha,\beta}$-semistable;
\item $H^{-1}_\beta(E)$ is $\nu_{\alpha,\beta}$-semistable and $H^0_\beta(E)$ is either 0 or supported in dimension 0.
\end{enumerate}
\end{Lem}

\begin{proof}
Consider the exact sequence
\[
0 \to H^{-1}_\beta(E)[1] \to E \to H^0_\beta(E) \to 0.
\]
in $\AA^{\alpha,\beta}(X)$. For $a\to +\infty$, we have
\[
\Re Z_{\alpha, \beta}^a\left(H^{-1}_\beta(E)[1]\right) = -a H^2 \ch_1^\beta(H^{-1}_\beta(E)) +
\mathrm{const} \to -\infty
\]
unless $H^{-1}_\beta(E) = 0$, and
\[
\Re Z_{\alpha, \beta}^a\left(H^0_\beta(E)\right) = a H^2 \ch_1^\beta(H^0_\beta(E)) - \ch_3^\beta(H^0_\beta(E))
\ge -\ch_3^\beta(H^0_\beta(E)).
\]
Their imaginary parts are constant, with $\Im Z_{\alpha, \beta}^a\left(H^0_\beta(E)\right) \neq 0$
unless $H^0_\beta(E)$ is supported in dimension zero. This means that $E$ is $\sigma_{\alpha,
\beta}^a$-unstable for $a \gg 0$ unless $H^{-1}_\beta(E) = 0$, or $H^0_\beta(E)$ is a zero-dimensional
torsion sheaf, or $H^0_\beta(E) = 0$.

In the limit $a \to +\infty$, we have $Z_{\alpha, \beta}^a \to \overline{Z}_{\alpha, \beta}$ up to
rescaling of the real part; this implies the $\nu_{\alpha, \beta}$-semistability of the cohomology
objects in both cases.
\end{proof}

We have already proved the analogue of Lemma \ref{lem:discriminantsignature}, as part of Lemma
\ref{lem:negativedefinite}. This also enables us to use the result from Appendix
\ref{app:SupportProperty}.

\begin{proof}[Proof of Lemma \ref{lem:RationalBGinDim3}]
Throughout the proof, we fix $\alpha$ and $\beta$.

If $E$ is strictly $\sigma_{\alpha,\beta}^a$-semistable, and if \eqref{eq:newBGdim3} holds for
all of the Jordan-H\"older factors $E_i$ of $E$, then by Lemma \ref{lem:blah}, it also holds for
$E$. We may therefore assume that $E$ is \emph{stable}.

We also notice that if $F\in\Coh^{\beta}(X)$ is $\nu_{\alpha,\beta}$-semistable, then
Conjecture \ref{con:BGgeneral} and Theorem \ref{thm:BGvariants} show that in particular, it satisfies
$Q_K^\beta(F) \ge 0$ for every $K > \alpha^2$.

We proceed by induction on $f(E):=H\ch_2^\beta(E) - \frac{\alpha^2 \, H^3}{2} \ch_0^\beta(E) = \Im Z_{\alpha, \beta}^a(E)$, which
is a non-negative function on $\AA^{\alpha, \beta}(X)$ with discrete values.

We fix $a_0 > \frac 16 \alpha^2$ and $K \in (\alpha^2, 6a_0)$.
Let $E$ be a $\sigma_{\alpha,\beta}^{a_0}$-stable object in $\AA^{\alpha,\beta}(X)$.

If $E$ remains $\sigma_{\alpha,\beta}^a$-semistable, for all $a > a_0$, then by Lemma \ref{lem:LVL}
either $E=H_\beta^0(E)$ is $\nu_{\alpha,\beta}$-semistable, or $H_\beta^{-1}(E)$ is $\nu_{\alpha,\beta}$-semistable and
$H^0(E)$ is either 0 or supported in dimension 0.
In the first case, we already pointed out above that $E$ satisfies \eqref{eq:newBGdim3}.
In the second case, $H^2\ch_1^\beta(E)=H^2\ch_1^\beta(H_\beta^{-1}(E)[1])<0$ and $\ch_3^\beta(H_\beta^0(E))\geq0$.
Therefore $\HDelta(E) = \HDelta(H_\beta^{-1}(E))$ and
\[
\HNablab(E) = \HNablab(H^{-1}_\beta(E)) - 6 H^2 \ch_1^\beta(H^{-1}_\beta(E)[1])
\ch_3^\beta(H^0_\beta(E)) \ge \HNablab(H^{-1}_\beta(E)).
\]
Since \eqref{eq:newBGdim3} holds for $H_\beta^{-1}(E)$, it holds also for $E$.

Otherwise, $E$ will be unstable for $a$ sufficiently big.
Every possibly destabilizing subobject or quotient $F$ has $f(F) < f(E)$ (since $f$ is non-negative,
and since the subcategory of objects $F \in \AA^{\alpha, \beta}(X)$ with $f(F) = 0$ has maximum
possible slope with respect to $Z_{\alpha, \beta}^a$ for all $a$).

Therefore they obey the induction assumption; since $K \in (\alpha^2, 6a_0) \subset (\alpha^2,
6a)$, this means that all these possible subobject or quotients satisfy \eqref{eq:newBGdim3} with
respect to our choice of $K$.
Since $Z_{\alpha, \beta}^a$ has negative definite kernel with respect to $Q_K^\beta$ for all $a \ge
a_0$, this is equivalent to a support property type statement, see Appendix
\ref{app:SupportProperty}. It follows that $E$ satisfies well-behaved wall-crossing along our path.
Hence, there will exist $a_1> a_0$ such that $E$ is strictly $\sigma_{\alpha,\beta}^{a_1}$-semistable.
But all the Jordan-H\"older factors $E_i$ of $E$ have strictly smaller $f$. Using the induction
assumption again, we see that they satisfy $Q_K^\beta(E_i) \ge 0$; therefore, we can again
apply Lemma \ref{lem:blah} to deduce the same claim for $E$.
\end{proof}

The combination of Lemma \ref{lem:skewtocubic}, Theorem \ref{thm:rationalconstruction} and Lemma
\ref{lem:RationalBGinDim3} together with Proposition \ref{prop:EffectiveDefo} leads to the following result: for each
tuple $\alpha, \beta, a, b$ as in Theorem \ref{thm:rationalconstruction} (in particular $\alpha,
\beta \in \Q$), we obtain an open subset $U(\alpha, \beta, a, b) \subset \Stab_H(X)$ of stability
conditions by deforming the pair $(Z_{\alpha, \beta}^{a, b}, \AA^{\alpha, \beta}(X))$. The
associated open subsets $\ZZ(U(\alpha, \beta, a, b))$ of central charges combine to cover the set
$\PPP$. To conclude the proof of Theorems \ref{thm:mainPPP} and \ref{thm:BGinDim3}, we need to show
that the sets $U(\alpha, \beta, a, b)$ glue to form a continuous family covering $\widetilde \PPP$.

This is done by the following analogue of Proposition \ref{prop:DeformingTiltStability}:

\begin{Prop}\label{prop:DeformingBridgelandStability}
There is a continuous family of Bridgeland stability conditions in ${\Stab}_H(X)$,
parameterized by the set
\[ (\alpha, \beta, a) \in \R_{>0}\times \R \times \R, \quad a > \frac 16 \alpha^2 \]
via
\[
(\alpha,\beta, a) \mapsto
	\sigma^a_{\alpha, \beta} := \left(Z_{\alpha, \beta}^a, \AA^{\alpha, \beta}(X)\right).
\]
\end{Prop}

Indeed, deformations of the central charge $Z_{\alpha, \beta}^{a, b}$ for $b \neq 0$ (while keeping
$\alpha, \beta, a$ fixed) do not change the heart, as modifying $b$ only affects the real part of
the central charge. Acting on these stability conditions by $\GL_2^+(\R)$ produces the entire set
$\widetilde \PPP$.

To prove Proposition \ref{prop:DeformingBridgelandStability}, we need a few preliminary results. We
will use the notion of a \emph{pre-stability condition}, which is a stability condition that does
not necessarily satisfy the support property; see Appendix \ref{app:SupportProperty}. The
first result already appears implicitly in \cite[Section 10]{Bridgeland:K3}.

\begin{Lem} \label{lem:thesame}
Assume that $\sigma_1 = (Z, \AA_1)$ and $\sigma_2 = (Z, \AA_2)$ are two pre-stability
conditions with the following properties:
\begin{enumerate}
\item Their central charges agree.
\item There exists a heart $\BB$ of a bounded t-structure such that each $\AA_i$ can be obtained as
a tilt of $\BB$:
\[
\AA_1, \AA_2 \subset \langle \BB, \BB[1] \rangle.
\]
\end{enumerate}
Then $\sigma_1 = \sigma_2$.
\end{Lem}

\begin{proof}
By \cite[Lemma 1.1.2]{Polishchuk:families-of-t-structures}, for $i=1,2$, $\AA_i$ is a tilt of $\BB$ with respect to the torsion pair
\[
\TT_i := \BB \cap \AA_i \quad \text{ and } \quad \FF_i := \BB \cap \AA_i[-1].
\]
We need to show that $\TT_1=\TT_2$ and $\FF_1=\FF_2$; in fact, since $\FF_i = \TT_i^\perp$, it is
enough to show $\TT_1 = \TT_2$. Observe that, since the
central charges agree, we have $\TT_2 \cap \FF_1 = \{0\} = \TT_1 \cap \FF_2$.

We let $T\in\TT_2$.  Consider the exact sequence in $\BB$
\[
0 \to T_1 \to T \to F_1 \to 0,
\]
with $T_1\in\TT_1$ and $F_1\in\FF_1$.
Since the torsion part of any torsion pair is closed under quotients, $F_1\in\TT_2$, contradicting
the observation above.
Hence, $T\in\TT_1$, and so $\TT_2\subseteq\TT_1$. The reverse inclusion follows similarly.
\end{proof}

\begin{Lem} \label{lem:Zpositive}
There exists a continuous positive function $\epsilon(\alpha, \beta, a) > 0$ with the following
property:
if $E \in \Coh^\beta(X)$ is $\nu_{\alpha, \beta}$-stable with
\[ \abs{\nu_{\alpha, \beta}(E)} < \epsilon(\alpha, \beta, a), \]
then $\Re Z_{\alpha, \beta}^a(E) > 0$.
\end{Lem}

\begin{proof}
We first apply Conjecture \ref{con:BGgeneral}, rewriting \eqref{eq:BGgeneral} as 
\begin{equation} \label{eq:firstestimate}
6 \ch_3(E) \le \alpha^2 H^2\ch_1^\beta(E) + 4 H \ch_2^\beta(E) \nu_{\alpha, \beta}(E).
\end{equation}
Now we apply Theorem \ref{thm:BGvariants}. First of all, we can rewrite
$\HDelta(E) \ge 0$ as
\[
\left(H^2 \ch_1^\beta\right)^2
+ \frac 1{\alpha^2} \left(H \ch_2^\beta - \frac{\alpha^2}2 H^3 \ch_0^\beta\right)^2
- \frac 1{\alpha^2} \left(H \ch_2^\beta + \frac{\alpha^2}2 H^3 \ch_0^\beta\right)^2 \ge 0.
\]
By assumption, 
\[ \abs{H \ch_2^\beta - \frac{\alpha^2}2 H^3 \ch_0^\beta} < \epsilon H^2 \ch_1^\beta \]
and therefore 
\[ \abs{H \ch_2^\beta + \frac{\alpha^2}2 H^3 \ch_0^\beta} < \sqrt{\alpha^2 +
\epsilon^2}H^2 \ch_1^\beta. \]
Summing up the last two equations we obtain
\[ \abs{H \ch_2^\beta} \le \frac{\epsilon + \sqrt{\alpha^2 + \epsilon^2}}{2}H^2 \ch_1^\beta.\]
Plugging this into \eqref{eq:firstestimate}, we obtain the desired claim.
\end{proof}

\begin{Lem} \label{lem:letstry}
We keep the notation as in the previous lemma.
If $E \in \Coh^\beta(X)$ is $\nu_{\alpha, \beta}$-stable with
$\abs{\nu_{\alpha, \beta}(E)} < \epsilon$, then
$E \in \PP_{\alpha, \beta}^{a}((-\frac 12, \frac 12))$.
\end{Lem}

\begin{proof}
We consider just the case $0 < \nu_{\alpha, \beta}(E)$; the opposite case follows from dual
arguments.

By construction we know $E \in \AA^{\alpha, \beta} = \PP_{\alpha, \beta}^{a}((0,1])$. Let $A$ be
the HN-filtration factor of $E$ with respect to $\sigma_{\alpha, \beta}^a$ and with the largest phase, and consider the associated
short exact sequence $A \into E \onto B$ in $\AA^{\alpha, \beta}$. The associated long exact cohomology
with respect to $\Coh^\beta(X)$ shows that
$A \in \Coh^\beta(X) \cap \AA^{\alpha, \beta} = \TT'_{\alpha, \beta}$; moreover,
there is a sequence $H^{-1}(B) \into A \to E$ exact on the left with 
$H^{-1}(B) \in \FF'_{\alpha, \beta}$. 

Now consider the slopes appearing in the Harder-Narasimhan filtration of $A$ for tilt-stability
with respect to $\nu_{\alpha, \beta}$. By standard arguments using the observations in the previous
paragraph, all these slopes lie in the interval $(0, \epsilon)$. Lemma \ref{lem:Zpositive}
then implies $\Re Z_{\alpha, \beta}^a(A) > 0$, and therefore $E \in \PP_{\alpha, \beta}^a((0, \frac
12))$ as we claimed.
\end{proof}

\begin{proof}[Proof of Proposition \ref{prop:DeformingBridgelandStability}]
Consider a stability condition $\sigma_0 = (Z_0, \PP_0) := \sigma_{\alpha_0, \beta_0}^{a_0}$.
Let $\epsilon := \epsilon(\alpha_0, \beta_0, a_0)$ be as in Lemmas \ref{lem:Zpositive} and
\ref{lem:letstry}. 
Consider 
$(\alpha, \beta, a)$ sufficiently close to $(\alpha_0, \beta_0, a_0)$ (which we will make
precise shortly). 
Let $\sigma_1 := \sigma_{\alpha, \beta}^{a}$, and let $\sigma_2 = (Z_{\alpha, \beta}^a, \PP_2)$ be the stability condition with
central charge $Z_{\alpha, \beta}^{a}$ obtained by deforming $\sigma_0$. We want to apply 
Lemma \ref{lem:thesame} with $\BB = \PP_0\Big(\Big(-\frac 12, \frac 12\Big]\Big)$.

By the support property for $\sigma_0$, and the analogous property for tilt-stability, we can
require ``sufficiently close'' to mean that:
\begin{itemize}
\item if $E$ is $\sigma_2$-stable of phase $\phi$, then 
$E \in \PP_0((\phi-\epsilon, \phi+\epsilon))$, and
\item the analogous statement for tilt-stability with respect to $\nu_{\alpha, \beta}$ and
$\nu_{\alpha_0, \beta_0}$, respectively. This means that
if $E \in \Coh^{\beta}(X)$ is $\nu_{\alpha, \beta}$-semistable, and if $A_1, \dots,
A_m$ are the Harder-Narasimhan filtration factors of $E$ for tilt-stability with respect to
$\nu_{\alpha_0, \beta_0}$, then the phases of $\overline{Z}_{\alpha_0, \beta_0}(A_i)$ differs 
by at most $\epsilon$ from the phase of $\overline{Z}_{\alpha, \beta}(E)$, 
\end{itemize}
The first assumption implies that
\[ \PP_2((0,1]) \subset \PP_0((-\epsilon, 1+ \epsilon]) \subset \langle \BB, \BB[1] \rangle. \]

The second assumption implies that if $E$ is tilt-stable with respect to $\nu_{\alpha, \beta}$
and $\nu_{\alpha, \beta} > 0$, then all HN filtration factors $A_i$ of $E$ with respect
to $\nu_{\alpha_0, \beta_0}$ satisfy $\nu_{\alpha_0, \beta_0}(A_i) > - \epsilon$. 
In case $\nu_{\alpha_0, \beta_0}(A_i) > 0$ this implies $A_i \in \AA^{\alpha_0, \beta_0} = 
\PP_0((0, 1])$. Otherwise, if $- \epsilon < \nu_{\alpha_0, \beta_0}(A_i) < 0$, then
Lemma \ref{lem:letstry} shows $A_i \in \PP_0\big(\big(-\frac 12, \frac 12\big]\big)$; overall we obtain
\[ E \in \PP_0\Big(\Big(-\frac 12, 1\Big]\Big)\subset \langle \BB, \BB[1] \rangle.\]
A similar argument implies that if $E$ is tilt-stable with $\nu_{\alpha, \beta}\le 0$,
then $E[1] \in \langle \BB, \BB[1] \rangle$. Combined, these two facts show that
$\AA^{\alpha, \beta} \subset \langle \BB, \BB[1] \rangle$.

We have verified all the assumptions of Lemma \ref{lem:thesame}, which implies
$\sigma_1 = \sigma_2$.
\end{proof}

Let us also mention the following property:

\begin{Prop}[{\cite[Proposition 2.1]{Dulip-Antony:I}}]\label{prop:StableSkyscraper}
Skyscraper sheaves are stable for all $\sigma \in \widetilde \PPP$.
\end{Prop}

\begin{proof}[Proof (sketch)]
Using the long exact cohomology sequence with respect to the heart $\Coh(X)$, one sees that $k(x)$ is a minimal
object of $\Coh^\beta(X)$: otherwise, there would be a short exact sequence
$E \into k(x) \onto F[1]$ in $\Coh^\beta(X)$ coming from a short exact sequence $F \into E \onto k(x)$ of sheaves; this is
a contradiction to $\mu_{H, \beta}(F) < 0$ and $\mu_{H, \beta}(E) \ge 0$.
Similarly, taking the long exact cohomology sequence with respect to
$\Coh^\beta(X)$ of short exact sequences in $\AA^{\alpha, \beta}(X)$, we see that $k(x)$ is a
minimal object of $\AA^{\alpha, \beta}(X)$.
\end{proof}

\section{The space of stability conditions on abelian threefolds}\label{sec:abelianspace}

In this section we prove the following:

\begin{Thm} \label{thm:thatsit}
Let $(X, H)$ be a polarized abelian threefold. Then
$\widetilde \PPP \into \Stab_H(X)$ is a connected component of the space of stability
conditions.
\end{Thm}

The fundamental reason behind Theorem \ref{thm:thatsit} is the abundance of projectively flat vector
bundles on abelian threefolds; their Chern classes are dense in the projectivization of the twisted
cubic $\CCC$.

Consider a slope $\mu = \frac pq \in \Q$ with $p, q$ coprime and $q > 0$. Then there exists a family
of simple vector bundles $E_{p/q}$ that are semi-homogeneous in the sense of Mukai, have slope ${\frac
pq}$ and Chern character
\[
\ch(E_{p/q}) =  r \cdot e^{\frac {pH}q}, 
\]
see \cite[Theorem 7.11]{Mukai:semi-homogeneous}.
They can be constructed as the push-forward of line bundles via an isogeny $Y \to X$
\cite[Theorem 5.8]{Mukai:semi-homogeneous}, and are slope-stable \cite[Proposition
6.16]{Mukai:semi-homogeneous}.

The above theorem is essentially based on the following result:
\begin{Prop}
The semi-homogeneous vector bundle $E_{p/q}$ is $\sigma$-stable for every $\sigma \in \widetilde{\PPP}$.
\end{Prop}
\begin{proof}
As mentioned above, $E_{p/q}$ is slope-stable.  By Corollary \ref{cor:Delta0}, either $E_{p/q}$ or
$E_{p/q}[1]$ is a $\nu_{\alpha, \beta}$-stable object of $\Coh^\beta(X)$ for all $\alpha > 0, \beta \in \R$.

Also observe that for all $K, \beta \in \R$, we have
\[
\HDelta(E_{p/q}) = \HNablab(E_{p/q}) = 0 \Rightarrow Q_K^\beta(E_{p/q}) = 0.
\]
The open subsets of $\PPP$ where the central charges are negative definite
with respect to $Q_K^\beta = K \HDelta + \HNablab$ for some $K, \beta$ form a covering of $\PPP$; by Proposition
\ref{prop:Q0remainsstable}, it is therefore enough to find a single stability condition $\sigma \in
\widetilde \PPP$ for which $E_{p/q}$ is $\sigma$-stable.

One can prove in general that $\nu_{\alpha, \beta}$-stable vector bundles are
$\sigma_{\alpha, \beta}^{a, b}$-stable for $a \gg 0$; but in our situation one can argue more easily as
follows.  Choose $\alpha, \beta$ with $\beta < \frac pq$ (and therefore $E_{p/q} \in \Coh^\beta(X)$)
and  $\nu_{\alpha, \beta}(E) = 0$. Then $E[1] \in \AA^{\alpha, \beta}(X)$ with
$\Im Z_{\alpha, \beta}^{a, b} = 0$ for all $a, b$, i.e. it has maximal possible slope; therefore it
is $\sigma_{\alpha, \beta}^{a, b}$-semistable.  By Lemma \ref{lem:convexcone}, it
must actually be strictly stable.
\end{proof}

\begin{proof}[Proof of Theorem \ref{thm:thatsit}]
Assume for a contradiction that there is a stability condition $\sigma = (Z, \PP) \in \partial \widetilde
\PPP$ in the boundary of $\widetilde \PPP$ inside $\Stab_H(X)$.
Since $\widetilde{\PPP} \to \PPP$ is a covering map, the central charge $Z$ must be in the boundary
$\partial \PPP$ of $\PPP \subset \Hom(\Lambda_H, \C)$; by definition, this means that there is a
point $(x^3, x^2y, \frac 12 xy^2, \frac 16 y^3)$ on the twisted cubic $\CCC$ that is contained in the kernel
of $Z$.

If $\mu := \frac yx = \frac pq$ is rational, then we observe that every
semi-homogeneous bundle $E_{p/q}$ is $\sigma$-semistable, because being $\sigma$-semistable is a closed condition on $\Stab_H(X)$. This is an immediate contradiction, as
$Z\left(E_{p/q}\right) = 0$.
Similarly, if $x = 0$, we get $Z(\OO_x) = 0$; yet skyscraper sheaves of
points are $\sigma$-semistable by \ref{prop:StableSkyscraper}.

Otherwise, if $\mu \in \R \setminus \Q$, consider a sequence $(p_n, q_n)$ with
\[ \lim_{n \to \infty} \frac{p_n}{q_n} = \mu, \]
let $E_n := E_{{p_n}/{q_n}}$, and let $r_n = \rk E_n$.  Then
\[ \lim_{n \to \infty} \frac 1{r_n} v_H(E_n) = \left(1, \mu, \frac 12 \mu^2, \frac 16 \mu^3\right) \]
and thus
\[
\lim_{n \to \infty} \frac{\abs{Z\left(v_H(E_n\right)}}{\| v_H(E) \|}
= \lim_{n \to \infty} \frac{\abs{Z\left(\frac 1{r_n} v_H(E_n)\right)}}
						{\left\| \frac 1{r_n} v_H(E) \right\|}
= \frac{\abs{Z\left(1, \mu, \frac 12 \mu^2, \frac 16 \mu^3\right)}}
				{\left\|\left(1, \mu, \frac 12 \mu^2, \frac 16 \mu^3\right)\right\|}
= 0.
\]
This is a contradiction to the condition that $\sigma$ satisfies the support property.
\end{proof}

\section{The space of stability conditions on some Calabi-Yau threefolds}\label{sec:CY}

Let $X$ be a projective threefold with an action of a finite group $G$.
In this section, we recall the main result of \cite{MMS:inducing}, which induces  stability
conditions on the $G$-equivariant derived category from $G$-invariant stability conditions on $X$;
similar results are due to Polishchuk, see \cite[Section 2.2]{Polishchuk:families-of-t-structures}.
We use it to construct stability conditions on Calabi-Yau threefolds that are
(crepant resolutions of) quotients of abelian threefolds, thus proving Theorems
\ref{thm:mainCYabelian}, \ref{maintheorem} and \ref{MAINTHM}.

\subsection*{The equivariant derived category}
We let $\Coh([X/G])$ be the abelian category of $G$-equivariant coherent sheaves on $X$,
and $\Db([X/G]):=\Db(\Coh([X/G]))$.
As explained in \cite{Elagin:equivariant}, the category $\Db([X/G])$ is equivalent to the category of the
$G$-equivariant objects in $\Db(X)$.

The \'etale morphism $f \colon X \to [X/G]$ of Deligne-Mumford stacks induces
a faithful pull-back functor
\[ f^* \colon \Db([X/G])\to\Db(X).
\]

Let $H\in\NS(X)$ be an ample $G$-invariant divisor class.
We consider the space $\Stab_H(X)$ of stability conditions on $\Db(X)$ with respect to the lattice $\Lambda_H$ as in Section
\ref{sec:newstability}; for $\Db([X/G])$ we use the same lattice, and the map
\[
v_H^G \colon K(\Db([X/G])) \to \Lambda_H, \qquad v_H^G(E):=v_H(f^*(E)).
\]
By mild abuse of notation, we will write $\Stab_H([X/G])$ for the space of stability
conditions on $\Db([X/G])$ satisfying the support property with respect to
$(\Lambda_H, v_H^G)$.  We will construct components of $\Stab_H([X/G])$ from
$G$-invariant components of $\Stab_H(X)$.

\subsection*{Inducing stability conditions}
Following \cite{MMS:inducing}, we consider
\[
\Stab_H^G(X) :=
\stv{\sigma\in{\Stab}_H(X)}{g^*\sigma=\sigma\mbox{, for any
}g\in G}.
\]
Here the action of $g$ on $\Stab_H(X)$ is given by
\[
g^* (Z, \AA) = \left(Z \circ (g^*)^{-1} , g^*(\AA)\right).
\]
For any $\sigma=(Z,\AA)\in \Stab_H^G(X)$, we define
\[
(f^*)^{-1}(\sigma):=(Z',\AA')
\]
where
\begin{align*}
\ & Z':=Z \circ v_H^G, \\
\ & \AA':=\{E\in\Db([X/G]) \colon f^*(E)\in\AA\}.
\end{align*}

\begin{Thm}[{\cite{MMS:inducing}}]\label{thm:inducing}
Let $(X, H)$ be a polarized threefold with an action by a finite group $G$ fixing the polarization.
Then $\Stab_H^G(X) \subset \Stab_H(X)$ is a union of connected components.

Moreover, the pull-back $f^*$ induces an embedding
\[
(f^*)^{-1} \colon \Stab_H^G(X) \into  \Stab_H([X/G])
\]
whose image is again a union of connected components.
\end{Thm}

\begin{proof}
The theorem is essentially a reformulation of Theorem 1.1 in \cite{MMS:inducing} but some subtle issues have to be clarified. First of all, Theorem 1.1 in \cite{MMS:inducing} deals with stability conditions whose central charge is defined on the Grothendieck group $K(X)$ rather than on the lattice $\Lambda_H$. On the other hand, the same argument as in \cite[Remark 2.18]{MMS:inducing} shows that all the results in \cite[Section 2.2]{MMS:inducing}, with the obvious changes in the statements and in the proofs, hold true if we consider pre-stability conditions as in Definition \ref{def:Bridgeland2} with respect to the lattice $\Lambda_H$. Thus we will freely quote the results there.

We now observe that if $\sigma$ is a $G$-invariant pre-stability condition on $\Db(X)$, then
$\sigma$ satisfies the support property with respect to $v_H$ if and only if
 $(f^*)^{-1}(\sigma)$ satisfies the support property with respect to $v_H^G$.
This is rather obvious, given the definition of $(f^*)^{-1}(\sigma)$ above, the fact that
$\Lambda_H$ is invariant under the action of $G$ and that the semistable objects in
$(f^*)^{-1}(\sigma)$ are the image under $f^*$ of the semistable objects in $\sigma$ (see
\cite[Theorem 1.1]{MMS:inducing}). Hence \cite[Proposition 2.17]{MMS:inducing} applies and $(f^*)^{-1}$
yields a well-defined and closed embedding.

It remains to point out that $\Stab_H^G(X)$ is a union of connected components of $\Stab_H(X)$. This is clear in view of the arguments in \cite[Lemma 2.15]{MMS:inducing} and, again, of the fact that $\Lambda_H$ is invariant under the action of $G$. Thus the image of $(f^*)^{-1}$ is a union of connected components as well.
\end{proof}

An immediate consequence of the results of Section \ref{sec:newstability} and Theorem \ref{thm:inducing} is the following, which completes the proof of Theorem \ref{maintheorem} (see also Examples \ref{ex:FreeAction} and \ref{ex:GenKummer} below):

\begin{Prop}\label{prop:equivariant}
Let $(X, H)$ be a smooth polarized threefold with an action of a finite group $G$ fixing the
polarization.  Assume that Conjecture \ref{con:BGgeneral} holds for $(X, H)$.
Then, given $\alpha, \beta \in \R$ and $\alpha, \beta, a, b$ satisfying \eqref{ineq:a},
the stability condition $\left(Z_{\alpha, \beta}^{a, b}, \AA^{\alpha, \beta}(X)\right)$ is in
$\Stab_H^G(X)$, and $(f^*)^{-1}(\Stab_H^G(X))$ is a non-empty union of connected components of
$\Stab_H([X/G])$.
\end{Prop}

\begin{proof}
Given Theorem \ref{thm:inducing}, the result will follow once we prove that $\left(Z_{\alpha,
\beta}^{a, b}, \AA^{\alpha, \beta}(X)\right)$ is in $\Stab_H^G(X)$.
Since slope-stability with respect to $H$ is preserved by the group action, we have $g^*
\Coh^\beta(X) = \Coh^\beta(X)$ for all $g \in G$. The same argument holds for tilt-stability, as
\[
\nu_{\alpha,\beta}(g^*E)=\nu_{\alpha,\beta}(E)
\]
for all $g\in G$ and $E\in\Coh^\beta(X)$; therefore $A^{\alpha, \beta}(X)$ is $G$-invariant as well.
Since the central charge $ Z_{\alpha, \beta}^{a, b}$ is similarly preserved by $G$, this shows the
claim.
\end{proof}

As an immediate consequence we get the following.

\begin{Cor}\label{cor:abinv}
Let $(X, H)$ be a polarized abelian threefold with an action of a finite group $G$ fixing the
polarization. Then $(f^*)^{-1}(\widetilde \PPP)$ is a connected component of $\Stab_H([X/G])$.
\end{Cor}

\begin{proof}
By Theorem \ref{thm:thatsit}, the open subset $\widetilde \PPP$ is a connected component of
$\Stab_H(X)$. By Proposition \ref{prop:equivariant}, we have that $\widetilde\PPP\cap\Stab_H^G(X)$
is not empty. Since $\Stab_H^G(X)$ is a union of connected components of $\Stab_H(X)$ (see Theorem
\ref{thm:inducing}), we get that $\widetilde\PPP$ is a connected component of $\Stab_H^G(X)$. Again
by Theorem \ref{thm:inducing}, we conclude that $(f^*)^{-1}(\widetilde \PPP)$ is a connected
component of $\Stab_H([X/G])$.
\end{proof}

\subsection*{Applications}

When the action of the finite group $G$ is free, the quotient $Y=X/G$ is smooth and $\Db(Y)\cong\Db([X/G])$.
In this case, an ample class $H$ on $X$ induces an ample class $H_Y$ on $Y$.
If we take $B$ on $X$ to be $G$-invariant as well, and write $B_Y$ for the induced class on $\NS(Y)_\R$, we then have, by Proposition \ref{prop:conjet}, that Conjecture \ref{con:strongBG} holds for $\nu_{\sqrt{3}\alpha H_Y, B_Y}$-stability on $Y$ if it holds for $\nu_{\sqrt{3}\alpha H, B}$-stability on $X$.

Here is a list of examples where $X$ is an abelian threefold and this discussion can be implemented,
concluding the proof of Theorems \ref{thm:mainCYabelian} and \ref{MAINTHM}.

\begin{Ex}\label{ex:FreeAction}
(i) A \emph{Calabi-Yau threefold of abelian type} is an \'{e}tale quotient $Y=X/G$ of an abelian
threefold $X$ by a finite group $G$ acting freely on $X$ such that the canonical line bundle of $Y$
is trivial and $H^1(Y,\C)=0$. In \cite[Theorem 0.1]{OguisoSakurai:quotient}, those Calabi-Yau
manifolds are classified; the group $G$ can be chosen to be $(\Z/2)^{\oplus 2}$ or $D_8$, and the
Picard rank of $Y$ is 3 or 2, respectively.  The following concrete example is usually referred to as Igusa's example (see Example 2.17 in \cite{OguisoSakurai:quotient}).
Take three elliptic curves $E_{1}$,
$E_{2}$ and $E_{3}$ and set $X = E_{1} \times E_{2} \times E_{3}$. Pick three non-trivial elements $\tau_{1}$, $\tau_2$ and $\tau_3$ in the $2$-torsion subgroups of
$E_{1}$, $E_{2}$ and  $E_{3}$, respectively. Then we define two automorphisms $a$ and $b$ of
$X$ by setting
\[
a(z_{1}, z_{2}, z_{3}) = (z_{1} + \tau_{1}, -z_{2}, -z_{3}) \quad \text{ and } \quad
b(z_{1}, z_{2}, z_{3}) = (-z_{1}, z_{2} + \tau_{2}, -z_{3} + \tau_{3}).
\]
By taking $G:=\langle a, b \rangle$, the quotient $Y=X/G$ is a Calabi-Yau threefold of abelian type.

(ii) Let $A$ be an abelian surface and let $E$ be an elliptic curve.
We write $X:= A \times E$.
Consider a finite group $G$ acting on $A$ and $E$, where the action on $E$ is given by translations.
Then the diagonal action on $X$ is free, but it may have non-trivial (torsion) canonical bundle.
The easiest example is by taking $A$ as the product $E_1\times E_2$ of two elliptic curve, and the action of $G$ only on the second factor so that $E_2/G\cong\P^1$.
Then $Y=E_1 \times S$, where $S$ is a bielliptic surface.
\end{Ex}

Let us now assume that $X$ is an abelian threefold, that $G$ acts faithfully,
and that the dualizing sheaf is locally trivial
as a $G$-equivariant sheaf.  By \cite{Mukai-McKay}, the quotient $X/G$ admits a crepant
resolution $Y$ with an equivalence $\Phi_{\mathrm{BKR}} \colon \Db(Y)\to \Db([X/G])$. By a slightly
more serious abuse of notation, we will continue to write $\Stab_H(Y)$ for the space of stability
conditions with respect to the lattice $\Lambda_H$ and the map $v_H^G \circ (\Phi_{\mathrm{BKR}})_* \colon K(Y) \to
\Lambda_H$.  By Corollary \ref{cor:abinv}, we obtain a connected component as
$\left(\Phi_{\mathrm{BKR}}\right)^* (f^*)^{-1}
\left(\widetilde \PPP\right) \subset \Stab_H(Y)$.

\begin{Ex}\label{ex:GenKummer}
We say that a Calabi-Yau threefold is of \emph{Kummer type} if it is obtained as a crepant
resolution of a quotient $X/G$  of an abelian threefold $X$.
Skyscraper sheaves will be semistable but not stable with respect to the stability conditions induced from $X$.
We mention a few examples.

(i) Let $E$ be an elliptic curve, and let $X=E\times E\times E$.
We consider a finite subgroup $G\subset\mathrm{SL}(3,\Z)$ and let it act on $X$ via the identification $X=\Z^3\otimes_\Z E$.
These examples were studied in \cite{AndreattaWisniewski:Kummer} and classified
in \cite{Donten:Kummer}; there are 16 examples, and $G$ has size at most 24.
The singularities of the quotient $X/G$ are not isolated.

(ii) Let $E$ be the elliptic curve with an automorphism of order $3$, and let $X=E\times E\times E$.
We can take $G=\Z/3\Z$ acting on $X$ via the diagonal action.
Then the crepant resolution $Y$ of $X/G$ is a simply connected rigid Calabi-Yau threefold containing
$27$ planes, see \cite[Section 2]{Beauville:remarksonc1zero}.

One can also take $G \subset (\Z/3\Z)^3$ to be the subgroup of order 9 preserving the volume form.
These examples were influential at the beginning of mirror symmetry, see
\cite{Batyrev-Borisov:generalizedMS} and references therein.

(iii) Let $X$ be the Jacobian of the Klein quartic curve.
The group $G=\Z/7\Z$ acts on $X$, and again the crepant resolution $Y$ of $X/G$ is a simply connected rigid Calabi-Yau threefold.

(iv) We can also provide easy examples involving three non-isomorphic elliptic curves $E_1$, $E_2$ and $E_3$. Indeed, take the involutions $\iota_i \colon E_i\to E_i$ such that $\iota_i(e)=-e$, for $i=1,2,3$, and set $G:=\langle \iota_1\times\iota_2\times\id_{E_3}, \iota_1\times\id_{E_2}\times\iota_3\rangle$. The quotient $(E_1\times E_2\times E_3)/G$ admits a crepant resolution $Y$ which is a Calabi-Yau threefold. This is a very simple instance of the so called \emph{Borcea-Voisin construction} (see \cite{Bor, Voi}). This yields smooth projective Calabi-Yau threefolds as crepant resolutions of the quotient $(S\times E)/G$, where $S$ is a K3 surface, $E$ is an elliptic curve and $G$ is the group generated by the automorphism $f\times \iota$ of $S\times E$, with $f$ an antisymplectic involution on $S$ and $\iota$ the natural involution on $E$ above. Example 2.32 in \cite{OguisoSakurai:quotient} is yet another instance of this circle of ideas.
\end{Ex}

\appendix

\section{Support property via quadratic forms}
\label{app:SupportProperty}

In this appendix, we clarify the relation between support property, quadratic inequalities for
Chern classes of semistable objects, and effective deformations of Bridgeland stability conditions.

\subsection*{Equivalent definitions of the support property}

Let $\DD$ be a triangulated category, for which we fix a finite
rank lattice $\Lambda$ with a surjective map $v \colon K(\DD) \onto \Lambda$.
We recall the main definition of \cite{Bridgeland:Stab} with a slight change of terminology: a
stability condition not necessarily satisfying the support property will be called a
\emph{pre-stability condition}:
\begin{Def}\label{def:Bridgeland2}
A pre-stability condition on $\DD$ is a pair $(Z, \PP)$ where
\begin{itemize}
\item the central charge $Z$ is a linear map $Z \colon \Lambda \to \C$, and
\item $\PP$ is a collection of full subcategories $\PP(\phi) \subset \DD$ for all $\phi \in \R$,
\end{itemize}
such that
\begin{enumerate}
\item $\PP(\phi + 1) = \PP(\phi)[1]$;
\item for $\phi_1 > \phi_2$, we have
$\Hom(\PP(\phi_1), \PP(\phi_2)) = 0$;
\item for $0 \neq E \in \PP(\phi)$, we the complex number $Z(v(E))$ is contained in the ray
$\R_{>0} \cdot e^{i \pi \phi}$; and
\item every $E$ admits an HN-filtration
\[
\xymatrix@C=.5em{
0_{\ } \ar@{=}[r] & E_0 \ar[rrrr] &&&& E_1 \ar[rrrr] \ar[dll] &&&& E_2
\ar[rr] \ar[dll] && \ldots \ar[rr] && E_{m-1}
\ar[rrrr] &&&& E_m \ar[dll] \ar@{=}[r] &  E_{\ } \\
&&& A_1 \ar@{-->}[ull] &&&& A_2 \ar@{-->}[ull] &&&&&&&& A_m \ar@{-->}[ull]
}
\]
with $A_i \in \PP(\phi_i)$ and $\phi_1 > \phi_2 > \dots > \phi_m$.
\end{enumerate}
\end{Def}
We write $\phi_{\sigma}^+(E) := \phi_1$ and $\phi_{\sigma}^-(E) := \phi_m$ for the maximal and
minimum phase appearing in the HN filtration.
The mass is defined by $m_\sigma(E) := \sum_{i=1}^m \abs{Z(A_i)}$.

Recall the definition of the ``support property'' introduced by Kontsevich and Soibelman:
\begin{Def}[{\cite[Section 1.2]{Kontsevich-Soibelman:stability}}] \label{def:SupportProperty}
Pick a norm $\| \blank \|$ on $\Lambda \otimes \R$. The pre-stability
condition $\sigma = (Z, \PP)$ satisfies the support property if there exists a constant $C > 0$ such
that for all $\sigma$-semistable objects $0 \neq E \in \Db(X)$, we have
\begin{equation} \label{eq:supportproperty}
\| v(E) \| \le C \abs{Z(v(E))}
\end{equation}
\end{Def}
This notion is equivalent to $\sigma$ being ``full'' in the sense of \cite{Bridgeland:K3}, see \cite[Proposition B.4]{localP2}.
The definition is quite natural: it implies that if $W$ is in an
$\epsilon$-neighborhood of $Z$ with respect to the operator norm on $\Hom(\Lambda_\R, \C)$ induced by
$\| \blank \|$ and the standard norm on $\C$, then $W(E)$ is in a disc of radius
$\epsilon C \abs{Z(E)}$ around $Z(E)$ for all semistable objects $E$; in particular, we can bound
the difference of the arguments of the complex numbers $Z(E)$ and $W(E)$.

Moreover, it is equivalent to the following
notion; we follow Kontsevich-Soibelman and also call it ``support property'':
\begin{Def} \label{def:SupportProperty2}
The pre-stability condition $\sigma = (Z, \PP)$ satisfies the support property if there exists a
quadratic form $Q$ on the vector space $\Lambda_\R$ such that
\begin{itemize}
\item the kernel of $Z$ is negative definite with respect to $Q$, and
\item for any $\sigma$-semistable object $E \in \Db(X)$, we have
\[ Q(v(E)) \ge 0.  \]
\end{itemize}
\end{Def}

\begin{Lem}[{\cite[Section 2.1]{Kontsevich-Soibelman:stability}}]\label{lem:KS}
Definitions \ref{def:SupportProperty} and \ref{def:SupportProperty2} are equivalent.
\end{Lem}
\begin{proof}
If $\sigma = (Z, \PP)$ satisfies Definition \ref{def:SupportProperty}, then the quadratic form
\[
Q(w) := C^2 \abs{Z(w)}^2 - \|w\|^2
\]
evidently satisfies both properties of Definition \ref{def:SupportProperty2}. Conversely, assume we
are given a quadratic form $Q$ as in Definition \ref{def:SupportProperty2}.
The non-negative quadratic form $\abs{Z(w)}^2$ is strictly positive on the set where
$-Q(w) \le 0$; by compactness of the unit ball, there exists a constant $C$ such that
\[
C^2 \abs{Z(w)}^2 - Q(w)
\]
is a positive definite quadratic form. Then $Z$ clearly satisfies \eqref{eq:supportproperty}
with respect to the induced norm on $\Lambda_\R$.
\end{proof}

\subsection*{Statement of deformation properties}

By $\Stab_\Lambda(\DD)$ we denote the space of stability conditions satisfying the support property with
respect to $(\Lambda,v)$. By the main result of \cite{Bridgeland:Stab}, the forgetful map
\[
\ZZ \colon \Stab_\Lambda(\DD) \to \Hom(\Lambda, \C), \quad (Z, \PP) \mapsto Z
\]
is a local homeomorphism.
The following result is the main purpose of this appendix:

\begin{Prop}\label{prop:EffectiveDefo}
Assume that $\sigma = (Z, \PP)\in \Stab_\Lambda(\DD)$ satisfies the support property
with respect to a quadratic form $Q$ on $\Lambda_\R$.
Consider the open subset of $\Hom(\Lambda, \C)$ consisting of central charges whose kernel is
negative definite with respect to $Q$, and let $U$ be the connected component containing $Z$.
Let $\UU \subset\Stab_\Lambda(\DD)$ be the connected component of the preimage $\ZZ^{-1}(U)$ containing
$\sigma$.
\begin{enumerate}
\item The restriction
$ \ZZ|_\UU \colon \UU \to U $
is a covering map.
\item Any stability condition $\sigma' \in \UU$ satisfies the support property with respect to the same
quadratic form $Q$.
\end{enumerate}
\end{Prop}
In other words, this proposition gives an effective version of Bridgeland's deformation result
\cite[Theorem 1.2]{Bridgeland:Stab}, and shows that Chern classes of semistable objects for varieties continue to satisfy the
same inequalities within this class of deformations.

\subsection*{The quadratic form and wall-crossing}
We start with the observation that the quadratic form is preserved by wall-crossing:
\begin{Lem} \label{lem:blah}
Let $Q$ be a quadratic form on $\Lambda_\R$. Assume that $\sigma = (Z, \PP)$ is a pre-stability
condition such that the kernel of $Z$ is negative semi-definite with respect to $Q$.
If $E$ is strictly $\sigma$-semistable with Jordan-H\"older factors $E_1, \dots, E_m$,
and if $Q(E_i) \ge 0$ for all $i =1, \dots, m$, then
$Q(E) \ge 0$.
\end{Lem}
\begin{proof}
Let $\H_E \subset \Lambda_\R$ be the half-space of codimension one given as the preimage
of the ray $\R_{\ge 0} \cdot Z(E)$, and let $\CC^+ \subset \H_E$ be the subset defined by
$Q \ge 0$. By the following Lemma, $\CC^+$ is a convex cone, implying the claim.
\end{proof}

\begin{Lem} \label{lem:convexcone}
Let $Q$ be a quadratic form an a real vector space $V$, and let $Z \colon V \to \C$ be a linear map
such that the kernel of $Z$ is semi-negative definite with respect to $Q$. Let $\rho$ be a ray in
the complex plane starting at the origin. Then the intersection
\[ \CC^+ = Z^{-1}(\rho) \cap \left\{Q(\blank) \ge 0\right\} \]
is a convex cone.

Moreover, if we assume that $Q$ has signature $(2, \dim V -2)$, and that the kernel of $Z$ is
negative definite, then any vector $w \in \CC^+$ with $Q(w) = 0$ generates an extremal ray of
$\CC^+$.
\end{Lem}
\begin{proof}
To prove convexity we just need to show that if $w_1, w_2 \in \CC^+$, then
$Q(w_1 + w_2) \ge 0$. According to the taste of the reader, this can either be seen by drawing a
picture of 2-plane $\Pi$ spanned by $w_1, w_2$---the only interesting case being where $Q|_{\Pi}$ has
signature $(1, 1)$---, or by the following algebraic argument.
Assume that $Q(w_1 + w_2) < 0$.  Since $w_1, w_2 \in Z^{-1}(\rho)$, there exists $\lambda > 0$ such
that $w_1 - \lambda w_2$ is in the kernel of $Z$. We therefore have
\[
Q(w_1 - \lambda w_2) \le 0, \quad
Q(w_1) \ge 0, \quad Q(w_1 + w_2) < 0, \quad \text{and} \quad Q(w_2) \ge 0.
\]
This configuration is impossible, since the quadratic function $f(x) := Q(w_1 + x w_2)$ would have too many
sign changes.

To prove the second statement, observe that under these stronger assumptions and for $w_1, w_2,
\lambda$ as above, we have $Q(w_1 - \lambda w_2) < 0$. This implies $Q(w_1+ w_2) >0$, from which the
claim follows.
\end{proof}

Before returning to the proof of Proposition \ref{prop:EffectiveDefo}, let us add one additional
consequence:
\begin{Prop} \label{prop:Q0remainsstable}
Assume that the quadratic form $Q$ has signature $(2, \rk \Lambda_H -2)$.
Let $U \subset \Stab_\Lambda(\DD)$ be a path-connected set of stability conditions that
satisfy the support property with respect to $Q$. Let $E \in \Db(X)$ be
an object with $Q(E) = 0$ that is $\sigma$-stable for some $\sigma \in U$. Then $E$ is
$\sigma'$-stable for all $\sigma' \in U$.
\end{Prop}
\begin{proof}
Otherwise there would be a wall at which $E$ becomes strictly semi-stable. However, by the previous
Lemma, $v_H(E)$ is an extremal ray of the cone $\CC^+$. Therefore, all the Jordan-H\"older factors
$E_i$ must have $v_H(E_i)$ proportional to $v_H(E)$, in contradiction to $E$ being strictly stable
for some nearby central charges.
\end{proof}

\subsection*{Proof of the deformation property}

In a sense, Lemma \ref{lem:blah} is the key observation in the proof of Proposition
\ref{prop:EffectiveDefo}; the remainder boils down to a careful application of local finiteness of
wall-crossing, and of the precise version of the deformation result proved by Bridgeland.

To this end, we need to recall the definition of the metric on $\Stab_\Lambda(\DD)$.

\begin{Def}[{\cite[Proposition 8.1]{Bridgeland:Stab}}] \label{def:Bridgelandmetric}
The following is a generalized metric on $\Stab_\Lambda(\DD)$:
\[
d(\sigma_1, \sigma_2) := \sup_{0 \neq E \in \DD}
\left\{ \abs{\phi_{\sigma_2}^-(E) - \phi_{\sigma_1}^-(E)},
\abs{\phi_{\sigma_2}^+(E) - \phi_{\sigma_1}^+(E)}, \abs{\log m_{\sigma_2}(E) - \log m_{\sigma_2}(E)}
\right\}
\]
\end{Def}
Bridgeland's proof of the deformation result
in fact proves the following stronger statement:
\begin{Thm}[{\cite[Sections 6 and 7]{Bridgeland:Stab}}] \label{thm:Bridgelanddeform}
Assume that $\sigma = (Z, \PP)$ is a stability condition on $\DD$, and let $C>0$ be a constant with
respect to which $\sigma$ satisfies the support property condition \eqref{eq:supportproperty}. Let
$\epsilon < \frac 18$, and consider the neighborhood $B_{\frac{\epsilon}C}(Z)$ of $Z$ taken with respect to the operator norm on
$\Hom(\Lambda, \C)$. Then there exists an open neighborhood $U \subset \Stab_\Lambda(\DD)$
containing $\sigma$, such that $\ZZ$ restricts to a homeomorphism
\[ \ZZ|_U \colon U \to B_{\frac{\epsilon}C}(Z). \]
Therefore, $\Stab_\Lambda(\DD)$ is a complex manifold; moreover, the generalized metric of
Definition \ref{def:Bridgelandmetric} is finite on every connected component of $\Stab_\Lambda(\DD)$.
\end{Thm}

\begin{proof}[Proof of Proposition \ref{prop:EffectiveDefo}]
Consider the subset $\VV \subset \UU$ of stability conditions that do \emph{not} satisfy the second
claim; we want to prove that $\VV$ is empty, thereby establishing the second claim.

Given $\sigma' \in \VV$, there exists a $\sigma'$-semistable object $E$ with
with $Q(v(E)) < 0$; by Lemma \ref{lem:blah}, we may assume that $E$ is stable. By openness of
stability of $E$, there exists a neighborhood of $\sigma'$ contained in $\VV$; therefore, $\VV
\subset \UU$ is open.

We claim that $\VV \subset \UU$ is also a closed subset; since $\UU$ is a manifold and $\VV \subset \UU$ is open,
it is enough to show that if $\sigma \colon [0,1] \to \UU$ is a piece-wise linear path with $\sigma(t) \in \VV$ for
$0 \le t < 1$, then $\sigma(1) \in \VV$. By the definition of $\VV$ and Lemma \ref{lem:blah} there
exists an object $E_0$ that is $\sigma(0)$-\emph{stable} with $Q(v(E_0)) < 0$. Since $\sigma(1)
\notin \VV$, there must be $0 < t_1 < 1$ such that $E_0$ is strictly semistable; applying Lemma
\ref{lem:blah} again, it must have a Jordan-H\"older factor $\sigma(t_1)$-stable factor $E_1$ with $Q(v(E_1)) < 0$. Proceeding
by induction, we obtain an infinite sequence $0=t_0<t_1 < t_2 < t_3 < \dots < 1$ of real numbers and objects
$E_i$ such that $E_i$ is $\sigma(t)$-stable for $t_i \le t < t_{i+1}$, strictly semistable
with respect to $\sigma(t_{i+1})$ (having $E_{i+1}$ as a Jordan-H\"older factor), and satisfies
$Q(v(E_i)) < 0$. This is a contradiction by Lemma \ref{lem:noinfinitesubobjects} below.

Therefore, since $\VV \subset \UU$ is both open and closed, and does not contain $\sigma$, it must be
empty.

It remains to prove the first claim. By Theorem \ref{thm:Bridgelanddeform}, it is enough to show
that there is a continuous function $C \colon U \to \R_{>0}$ such that every $\sigma \in \UU$
satisfies the support property with respect to $C(\ZZ(\sigma))$. This is evident from the second
claim and the proof of Lemma \ref{lem:KS}.
\end{proof}

\begin{Lem} \label{lem:noinfinitesubobjects}
Let $\sigma \colon [0,1] \to \Stab_\Lambda(\DD)$ be a piece-wise linear path in the space of stability condition
satisfying the support property. Assume there is a sequence $0 = t_0 < t_1 < t_2 < \dots < 1$ of real
numbers and a sequence of objects $E_0, E_1, E_2, \dots$ with the following properties:
\begin{itemize}
\item $E_i$ is $\sigma(t_i)$-\emph{stable}.
\item $E_{i}$ is $\sigma(t_{i+1})$-semistable, and $E_{i+1}$ is one of its Jordan-H\"older factors.
\end{itemize}

Such a sequence always terminates.
\end{Lem}

\begin{proof}
Assume we are given an infinite such sequence.  Let
$d_i := d\left(\sigma(t_i), \sigma(t_{i+1})\right)$; the assumptions imply that $\sigma$ is a path of bounded length, and
hence that
\[ D:= \sum_i d_i < +\infty.\]
On the other hand, if we write $Z_i$ for the central charge of $\sigma(t_i)$, then
\[
\abs{Z_{i+1}\left(E_{i+1}\right)} <
\abs{Z_{i+1}\left(E_{i}\right)} = m_{\sigma(t_{i+1})}(E_i)
\le e^{d_i} m_{\sigma(t_{i})}(E_i) = e^{d_i}  \abs{Z_{i}\left(E_{i}\right)};
\]
using induction we deduce that the mass of all objects $E_i$ is bounded:
\[
m_{\sigma(0)}(E_i) \le e^D m_{\sigma(t_i)}(E_i)  = e^D
\abs{Z_{i}\left(E_{i}\right)} \le e^{2D} \abs{Z_{0}\left(E_{0}\right)}.
\]
By \cite[Section 9]{Bridgeland:K3}, this implies that there is a locally finite collection of walls
of semistability for all $E_i$. Since our path is compact, it intersects only finitely many walls;
since it is piece-wise linear, it intersects every wall only finitely many times.
\end{proof}

\section{Deforming tilt-stability}\label{app:deformingtilt}

The purpose of this appendix is to establish rigorously the deformation and wall-crossing properties
of tilt-stability, in particular correcting \cite[Corollary 3.3.3]{BMT:3folds-BG}. This will lead to
variants of the results of Appendix \ref{app:SupportProperty} in this context. We assume that the
reader of this appendix is familiar with the notion of tilt-stability as reviewed in Sections
\ref{sec:TiltStability} and \ref{sec:classicalBG}, as well as with the proof of Bridgeland's
deformation result for stability conditions in \cite[Sections 6 and 7]{Bridgeland:Stab}.

Let $X$ be a smooth projective threefold with polarization $H$; the role of $\Lambda$ and $ v$ in
the previous appendix will be played by
\[
\begin{split}
&\Lambda = H^0(X, \Z) \oplus \NS(X)_\Z \oplus \frac 12 \Z \\
&v_H \colon K(X) \to \Lambda, \quad v_H(E) = \left(\ch_0(E), \ch_1(E), H \ch_2(E)\right).
\end{split}
\]

We will use a variant of the notion of ``weak stability'' of \cite{Toda:PTDT}, adapted
to our situation:
\begin{Def}
A very weak stability condition on $X$ is a pair
$\sigma = (Z, \AA)$, where $\AA$ is the heart of a bounded t-structure on $\Db(X)$,
and $Z \colon \Lambda \to \C$ is a group homomorphism such that
\begin{itemize}
\item
$Z$ satisfies the following weak positivity criterion for every $E \in \AA$:
\[ \Re Z(v_H(E)) \ge 0 \quad \text{and} \quad \Re Z(v_H(E)) = 0 \Rightarrow \Im Z(v_H(E)) \ge 0 \]
\item If we let $\nu_{Z, \AA} \colon \AA \to \R \cup {+\infty}$ be the induced slope function,
then HN filtrations exist in $\AA$ with respect to $\nu_{Z, \AA}$-stability.
\end{itemize}
\end{Def}
By induced slope function we mean that $\nu_{Z, \AA}(E)$ is the usual slope $\frac{\Im}{\Re}$ of the complex number
$Z(v_H(E))$ if its real part is positive, and $\nu_{Z, \AA}(E) = +\infty$ if $Z(v_H(E))$ is purely
imaginary or zero.
The crucial difference to a Bridgeland stability condition is that $Z(v_H(E)) = 0$ is allowed for
non-zero objects $E \in \AA$.

Given a very weak stability condition, one can define a slicing
$\PP = \{\PP(\phi) \subset \Db(X) \}_{\phi \in \R}$ just as in the case of a proper stability
condition constructed from a heart of a t-structure: for $-\frac 12 < \phi \le \frac 12$, we let
$\PP(\phi) \subset \AA$ be the subcategory of $\nu_{Z, \AA}$-semistable objects with slope
corresponding to the ray $\R_{>0} \cdot e^{i \pi \phi}$; this gets extended to all
$\phi \in \R$ via $\PP(\phi + n) = \PP(\phi)[n]$ for $n \in \Z$.

This allows one to define a topology on the set of very weak stability conditions; it is the
coarsest topology such that the maps $\sigma \mapsto Z$ and $\sigma \mapsto \phi_{\sigma}^\pm(E)$
are continuous, for all $E \in \Db(X)$. Our first goal is to show tilt-stability conditions vary
continuously; note that we use a slightly different normalization of the central charge than
in Section \ref{sec:TiltStability}:
\begin{Prop}\label{prop:DeformingTiltStability}
There is a continuous family of very weak stability conditions parameterized by
$ \R_{>0} \times \NS(X)_\R$ given by
\[ \left(\alpha, B \right) \mapsto \left(\overline{Z}_{\alpha, B}, \Coh^{H, B}(X)\right)\]
where
\[ \overline{Z}_{\alpha, B} = H^2 \ch_1^B + i \left(H \ch_2^B - \frac 12 \alpha^2 H^3
\ch_0^B\right).\]
\end{Prop}

For rational $B$, this stability condition can be constructed by proving directly that the pair
$\left(\overline{Z}_{\alpha, B}, \Coh^{H, B}(X)\right)$ admits Harder-Narasimhan filtrations, see
\cite[Lemma 3.2.4]{BMT:3folds-BG}. We will extend this to arbitrary $B$ by deformations, and show
simultaneously that these deformations glue to give a single family of very weak stability
conditions.

Let us first indicate the key difficulty that prevents us from applying the methods of
\cite[Sections 6 and 7]{Bridgeland:Stab} directly.
Let $I$ be a small interval containing $\frac 12$; then the \emph{quasi-abelian category}
$\PP(I)$ is not Artinian:
if $x \in X$ lies on a curve $C \subset X$, then $\dots \into \OO_C(-2x) \into \OO_C(-x) \into \OO_C$
is an infinite chain of strict subobjects of $\OO_C$ in $\PP(\frac 12) \subset \PP(I)$. Therefore,
the proof of \cite[Lemma 7.7]{Bridgeland:Stab} does not carry over.

We now explain how to circumvent this problem.  Fix $\alpha, B$ with $B$
rational; we will use $Z := \overline{Z}_{\alpha, B}$
for the corresponding central charge.  By the rational case of Theorem \ref{thm:BGvariants}, proved in Section
\ref{sec:classicalBG}, the central charge $Z$ satisfies the support property.\footnote{Note that
definitions \ref{def:SupportProperty} and \ref{def:SupportProperty2} both apply verbatim in this
situation: they allow for $Z(v_H(E)) = 0$ for a stable object $E$ if and only if $v_H(E) = 0$.} Let
$C > 0$ be the constant appearing in the support property; we also write $\PP$ for the associated
slicing.

Now consider a central charge $W := \overline{Z}_{\alpha', B'}$, where $\alpha',B'$ are sufficiently close to $\alpha,B$ such that
$W$ satisfies $\|W - Z\| < \frac{\epsilon}C$, for some sufficiently small
$\epsilon > 0$; recall that this implies that the phases of
$\sigma$-semistable object change by at most $\epsilon$.
We choose $\epsilon < \frac 18$ and small enough such that $\abs{H^2(B'-B)} < \alpha H^3$ is automatically satisfied.
For simplicity we also assume that $H^2(B' - B) < 0$; the other case can be dealt with analogously.

Let $I = (a, b)$ be a small interval with  $ a +\epsilon < \frac 12 < b - \epsilon$; the key problem is to construct
Harder-Narasimhan filtrations of objects in $\PP(I)$ with respect to $W$. Our first observation is that
due to our assumption $H^2(B' - B) < 0$, central charges of objects in $\PP(I)$ can only ``move to
the left''; this is again based on the Bogomolov-Gieseker inequality for $\sigma$-stability:

\begin{Lem} \label{lem:leftmoving}
If $E \in \PP((a, b))$ is $\sigma$-semistable with $\Re Z(E) < 0$, then also $\Re W(E) < 0$.
\end{Lem}
Note that the assumption is equivalent to $E \in \Coh^{H, B}(X)[1] \cap \PP((a, b)) = \PP\left(\left(\frac 12, b\right)\right)$.

\begin{proof}
By assumption we have $\Re Z(E) = H^2 \ch_1^B (E)< 0$ and
\[ \Re W(E) = \Re Z(E) - H^2(B'-B) \ch_0(E).\]
The case $\ch_0(E) \le 0$ is trivial due to the assumption $H^2(B'-B) < 0$. Otherwise, note that
$\Im Z(E) \ge 0$ implies
\[
2 H \ch_2^B(E) H^3 \ch_0(E) - \alpha^2 \left(H^3 \ch_0(E)\right)^2 \ge 0.
\]
By using Theorem \ref{thm:BGvariants}, applied to the rational class $B$, we also have
\[
2 H \ch_2^B(E) H^3 \ch_0(E) - \alpha^2 \left(H^3 \ch_0(E)\right)^2
\le \left(H^2 \ch_1^B(E)\right)^2- \alpha^2 \left(H^3 \ch_0(E)\right)^2.
\]
Therefore, we deduce
\[
H^2 \ch_1^B (E) \le -\alpha H^3 \ch_0(E).
\]
Using $\abs{H^2(B'-B) } < \alpha H^3$, this implies the claim.
\end{proof}

As in \cite[Section 7]{Bridgeland:Stab}, we define the set of semistable objects
$\QQ(\phi)$ to be objects of $\PP((\phi-\epsilon, \phi+\epsilon))$ that are $W$-semistable in a
slightly larger category, e.g.  in $\PP((\phi - 2 \epsilon, \phi + 2\epsilon))$. The key lemma
overcoming the indicated difficulty above is the following:

\begin{Lem} \label{lem:coarseHN}
Given $E \in \PP((a, b))$, there exists a filtration $ 0 = E_0 \into E_1 \into E_2 \into E_3$
such that
\begin{itemize}
\item $E_1 \in \Coh^{H, B'}(X)[1]$ and $E_1$ has no quotients $E_1 \onto N$ in $\PP((a, b))$ with
$\Re W(N) \ge 0$;
\item $E_2/E_1 \in \Coh^{H, B'}(X)$ is $W$-semistable in $\PP((a, b))$ with $\Re W(E_2/E_1) = 0$;
\item $E_3/E_2 \in \Coh^{H, B'}(X)$ and $E_3/E_2$ has no subobjects $M \into E_3/E_2$ in $\PP((a, b))$ with
$\Re W(M) \le 0$.
\end{itemize}
\end{Lem}

\begin{proof}
The t-structure associated to
$\Coh^{H, B}(X)$ gives a short exact sequence $E' \into E \onto E''$ in $\PP((a, b))$ with $E' \in \Coh^{H, B}(X)[1]$
and $E'' \in \Coh^{H, B}(X)$. Any quotient $E' \onto N$ would necessarily satisfy
$N \in \PP((\frac 12, b))$; by Lemma \ref{lem:leftmoving}, this implies $\Re W(N) < 0$.
Thus, given a filtration as in the claim for $E''$, its preimage in $E$ will still satisfy all the
claims.

We may therefore assume $E \in \Coh^{H, B}(X)$.
Note that $\Coh^{H, B'}(X)$ can be obtained as a tilt of $\Coh^{H, B}(X)$: there exists a
a torsion pair
\[ \TT = \Coh^{H, B'}(X)[1] \cap \Coh^{H, B}(X),\qquad \FF = \Coh^{H, B'}(X) \cap \Coh^{H, B}(X). \]
Moreover,
\[T \in \TT \Rightarrow \Re W(T) < 0 \quad \text{and} \quad F \in \FF \Rightarrow \Re W(T) \ge 0. \]
Let $E_1 \into E \onto F$ be the short exact sequence associated to $E$ via this torsion pair. Since
$\TT$ is closed under quotients, $E_1$ satisfies all the claims in the lemma; similarly, $B$
only has subobjects with $\Re W(\blank) \ge 0$.

The existence of $E_2$ now follows from the fact that $\Coh^{H, B'}(X)$ admits a torsion pair
whose torsion part is given by objects with $\Re W(\blank) = 0$; this is shown in the first
paragraph of the proof of \cite[Lemma 3.2.4]{BMT:3folds-BG}, which does not use any rationality
assumptions.
\end{proof}

The existence of Harder-Narasimhan filtrations of $E_1$ and $E_3/E_2$ can now be proved with the same methods as
in \cite[Section 7]{Bridgeland:Stab}; the same goes for any $E \in \PP((a, b))$ when $(a, b)$ is
an interval not intersecting the set $\frac 12 + \Z$; this is enough to conclude the existence of
HN filtrations for arbitrary $E \in \Db(X)$, see the arguments at the end of Section 7 in
\cite{Bridgeland:Stab}. Similar arguments apply in the case $H^2(B'-B) > 0$.

We have thus proved the claim that the tilt-stability condition $\sigma$ deforms to a very weak
stability condition $\sigma'$ with central charge $W$. Moreover, by the construction in Lemma
\ref{lem:coarseHN}, its associated t-structure is exactly $\Coh^{H, B'}(X)$; this finishes the proof
of Proposition \ref{prop:DeformingTiltStability}.

Let us also observe that for $\phi \in \frac 12 + \Z$, the subcategory
$\PP(\phi) \cap \stv{E \in \Db(X)}{v_H(E) = 0}$ is unchanged under deformations: it consists of
0-dimensional torsion sheaves, shifted by $\phi - \frac 12$. These are the only semistable objects
with central charge equal to zero; we will use this fact to show that tilt-stability conditions
satisfy well-behaved wall-crossing:

\begin{Prop} \label{prop:tiltwallcrossing}
Fix a class $c \in \Lambda$.
There exists a wall-and-chamber structure given by a locally finite set
of walls in  $\R_{>0} \times \NS(X)_\R$ such that for an object $E$ with $v_H(E) = c$, tilt-stability is unchanged
as $(\alpha, B)$ vary within a chamber. Each of the walls is locally given by one of the following
conditions on $Z = \overline{Z}_{\alpha, B}$:
\begin{enumerate}
\item \label{enum:standardseqence}
$Z(F)$ is proportional to $Z(E)$ for some destabilizing subobject $F \into E$ with
$v_H(F) \neq 0 \neq v_H(E/F)$, or
\item $Z(E)$ is purely imaginary (if there exists a subobject or quotient $F$ with $v_H(F) = 0$).
\end{enumerate}
\end{Prop}
\begin{proof}
As indicated above, the second type of walls corresponds to the case where $E$ has a shift of a
zero-dimensional torsion sheaf as a subobject or quotient. Otherwise, any possibly destabilizing
short exact sequence $F \into E \onto E/F$ must have the properties given in
\eqref{enum:standardseqence}, to which the usual arguments (e.g. in \cite[Section 9]{Bridgeland:K3})
based on support property apply.
\end{proof}

This allows us to complete the proof of the Bogomolov-Gieseker type inequalities:
\begin{proof}[Proof of Theorem \ref{thm:BGvariants}, case $H^2B$ non-rational]
Consider a $\nu_{\omega,B}$-semistable object $E$.  We may assume have $H^2\ch_1^B(E)\neq0$.
Using Lemma \ref{lem:halfspace}, we can assume that $E$ is in
fact $\nu_{\omega, B}$-stable.
By Proposition \ref{prop:tiltwallcrossing}, there is an open chamber in $\R_{>0} \times \NS(X)_\R$
in which $E$ is tilt-stable; this chamber contains points with rational $B$; therefore, our claim
$\HDeltaBC(E) \ge 0$ follows from case $H^2B$ rational proved in Section \ref{sec:classicalBG}.
\end{proof}


\begin{Rem}
(a) Alternatively, the statements of this appendix could be proved via the relation of tilt-stability to a
certain polynomial stability condition (in the sense of \cite{large-volume}); see Sections 4 and 5
of \cite{BMT:3folds-BG}, in particular Proposition 5.1.3. The advantage is that the slicing
associated to this polynomial stability condition is locally finite.

(b) Let us also explain precisely the problem with the statement of \cite[Corollary 3.3.3]{BMT:3folds-BG}: if we
allow arbitrary deformations of $\omega \in \NS(X)_\R$, rather than just those proportional to a
given polarization $H$, we would need to prove the support property for tilt-stable objects with
respect to a non-degenerate quadratic form on the lattice
\[
\Lambda = H^0(X) \oplus \NS(X) \oplus \frac 12 N_1(X), \quad
v(E) = \left(\ch_0(E), \ch_1(E), \ch_2(E)\right).
\]
However, none of the variants of the classical Bogomolov-Gieseker inequality discussed in Section
\ref{sec:classicalBG} give such a quadratic form, as they only depend on $H \ch_2$ rather than
$\ch_2$ directly.
\end{Rem}

\bibliography{abelian}                      
\bibliographystyle{alphaspecial}     

\end{document}